\documentclass[a4paper,12pt,reqno]{amsart}
\usepackage{graphicx}

\usepackage{color}

\usepackage{amsmath,amstext,amssymb,amsopn,amsthm}
\usepackage{url}

\usepackage[margin=30mm]{geometry}
\usepackage{eucal,mathrsfs,dsfont}
\usepackage{soul}

\newtheorem{theorem}{Theorem}[section]
\newtheorem{corollary}[theorem]{Corollary}
\newtheorem{lemma}[theorem]{Lemma}
\newtheorem{proposition}[theorem]{Proposition}

\theoremstyle{definition}

\newtheorem{definition}[theorem]{Definition}

\theoremstyle{remark}
\newtheorem{remark}[theorem]{Remark}
\newtheorem{example}[theorem]{Example}

\definecolor{mr}{rgb}{0.1,0.2,0.7}

\newcommand{\tadeusz}{\color{magenta}}
\newcommand{\normal}{\color{black}}

\newcommand{\calA}{\mathcal{A}}
\newcommand{\calB}{\mathcal{B}}

\newcommand{\calF}{\mathcal{F}}
\newcommand{\tcalF}{\tilde{\mathcal{F}}}
\newcommand{\calFd}{\mathcal{F}^{(d)}}
\newcommand{\tcalFd}{\tilde{\mathcal{F}}^{(d)}}
\newcommand{\tcalFdd}{\tilde{\mathcal{F}}^{(d+2)}}

\newcommand{\V}{L}

\newcommand{\R}{\mathds{R}}

\newcommand{\N}{{\mathds{N}}}

\newcommand{\RR}{\mathrm{I\kern-0.20emR}}
\newcommand{\D}{\mathrm{d}\kern0.2pt}

\newcommand{\ve}{{\varepsilon}}

\DeclareMathOperator{\supp}{supp}
\DeclareMathOperator{\dist}{dist}

\title[Gradient estimates of harmonic functions for L{\'e}vy processes]{Gradient estimates of harmonic functions and transition densities for L{\'e}vy processes}
\author[T. Kulczycki]{Tadeusz Kulczycki}
\author[M. Ryznar]{Micha{\l} Ryznar}
\thanks{The research was supported in part by NCN grant no. 2011/03/B/ST1/00423.}
\address{Institute of Mathematics and Computer Science, Wroc{\l}aw University of Technology, Wyb. Wyspia{\'n}skiego 27, 50-370 Wroc{\l}aw, Poland.}
\email{Tadeusz.Kulczycki@pwr.wroc.pl}
\email{Michal.Ryznar@pwr.wroc.pl}

\begin{document}
\begin{abstract}
We prove gradient  estimates for harmonic functions with respect to a $d$-dimensional  unimodal pure-jump L\'evy process under some mild assumptions on the  density of its L\'evy measure. These assumptions allow for a construction of an unimodal L\'evy process in $\R^{d+2}$ with the same characteristic exponent as the original process. The relationship between the two processes provides a fruitful source of gradient estimates of transition densities. We also construct another process called a difference process which is very useful in the analysis of differential properties of harmonic functions. Our results extend the gradient estimates from \cite{BKN2002} to a wide family of isotropic pure-jump process including a large class of subordinate Brownian motions.  
\end{abstract}

\maketitle

\section{Introduction}
The main purpose of this paper is to investigate the growth properties of a gradient of functions which are harmonic with respect to some isotropic L{\'e}vy processes in $\R^d$. Another aim is to obtain gradient estimates of transition densities of these processes. Our main result concerning a gradient of harmonic functions is the following theorem.

\begin{theorem}
\label{main}
Let $X$ be an isotropic L{\'e}vy process in $\R^d$ satisfying assumptions (A) (formulated below). Let $D \subset \R^d$ be an open, nonempty set and let $f: \R^d \to [0,\infty)$ be a function which is harmonic with respect to $X$ in $D$. Then $\nabla f(x)$ exists for any $x \in D$ and we have
\begin{equation}
\label{maineq}
|\nabla f(x)| \le c\frac{f(x)}{\delta_D(x) \wedge 1}, \quad \quad x \in D, 
\end{equation} 
where $\delta_D(x) = \dist(x,\partial D)$ and $c$ is a constant depending only on the process $X_t$.
\end{theorem}
The proof of this result is based on a new observation about gradient of transition densities for L{\'e}vy processes (see Theorem \ref{dertransition}) and a new concept of a difference process (see Section 4). It also uses recent results of P. Kim and A. Mimica \cite{KM2012} and K. Bogdan, T. Grzywny and M. Ryznar \cite{BGR2013}, \cite{BGR2013_1}, \cite{G2013}. The dependence of the constant $c$ in Theorem \ref{main} on the process $X_t$  will be further clarified in Remark \ref{constants}.

\begin{remark}
We  use a convention that for a radial function $f:\R^d \to \R$ we  write $f(x) = f(r)$, if $x \in \R^d$ and $|x| = r$.
\end{remark}

{\bf{Assumptions (A).}}

(H0) $X = (X_t, t \ge 0)$ is a pure-jump isotropic L{\'e}vy process in $\R^d$ with the characteristic exponent $\psi$ (i.e. $E^0 e^{i\xi X_t} = e^{-t\psi(\xi)}$). We assume that its L{\'e}vy measure is infinite and has the density $\nu(x) = \nu(|x|)$.

(H1) $\nu(r)$ is nonincreasing, absolutely continuous such that $-\nu'(r)/r$ is nonincreasing,  satisfies $\nu(r) \le a_1 \nu(r+1), r \ge 1$ and $\nu(r) \le a_1 \nu(2r), 0<r \le 1$ for some constant $a_1$.

(H2) The scale invariant Harnack inequality holds for the process $X$ (for the precise definition see Preliminaries).

\vskip 6pt

The derivative $\nu'(r)$ is understood as a function (defined a.e. on $(0, \infty)$) such that  $\nu(r) = -\int_r^\infty\nu'(\rho) d\rho,\,  r>0 $.
In fact, under the assumption that $-\nu'(r)/r$ is nonincreasing on the set where it is
defined, we can always take a version which is well defined for each point $r > 0$ and
$-\nu'(r)/r$ is nonincreasing on $(0, \infty)$. Throughout the whole paper we use that meaning
of $\nu'(r)$.
Note also that if $\nu(r)$ is convex then $-\nu'(r)/r$ is nonincreasing (in the above sense). 

Observe that the condition (H2) is also necessary for the gradient estimate (\ref{maineq}), since (\ref{maineq}) implies the scale invariant Harnack inequality.

The next result exhibits  some examples of classes of processes which satisfy assumptions (A). Before its formulation  we  introduce the definition of {\it{weak lower scaling condition}} (cf. \cite{BGR2013}). Let $\varphi$ be a non-negative, non-zero function on $[0,\infty)$. We say that $\varphi$ satisfies {\it{weak lower scaling condition}} WLSC($\underline{\alpha},\theta_0,\underline{C}$) if there are numbers $\underline{\alpha} > 0$, $\theta_0 \ge 0$ and $\underline{C} > 0$ such that
$$
\varphi(\lambda \theta) \ge \underline{C} \lambda^{\underline{\alpha}} \varphi(\theta), \quad \text{for} \quad \lambda \ge 1, \, \theta \ge \theta_0.
$$

\begin{proposition}
\label{classes} 
Let us consider the following conditions:

\vskip 2pt

{\bf{Assumptions $(\mathbf{A1})$.}} We assume (H0), (H1) and 

(H3) $\psi$ satisfies WLSC($\underline{\alpha},\theta_0,\underline{C}$).

\vskip 2pt

{\bf{Assumptions $(\mathbf{A2})$.}} We assume 

(H4) $X = (X_t, t \ge 0)$ is a subordinate Brownian motion, that is $X_t = B_{S_t}$, where $B = (B_t, t \ge 0)$ is the Brownian motion in $\R^d$ (with the generator $\Delta$) and $S = (S_t, t \ge 0)$ is a subordinator independent of $B$. The L{\'e}vy measure of $S$ is infinite.

(H5) The potential measure of $S$ has a decreasing density.

(H6) The L{\'e}vy measure of $S$ is infinite and has a decreasing density $\nu_S(r)$. 

(H7) There exist constants $\delta \in (0,1]$, $\theta_0 > 0$, $\overline{C}$ such that the Laplace exponent $\phi$ of $S$ satisfies
$$
\frac{\phi'(\lambda \theta)}{\phi'(\theta)} \le \overline{C} \lambda^{-\delta}, \quad \text{for} \quad \lambda \ge 1, \, \theta \ge \theta_0.
$$

(H8) The density of the L{\'e}vy measure $\nu(x) = \nu(|x|)$ of the process $X$ satisfies $\nu(r) \le a_1 \nu(r+1), r \ge 1$, for some constant $a_1 \ge 1$.

(H9) $d \ge 3$.

\vskip 2pt

{\bf{Assumptions (A3).}} We assume (H4), (H7) and 

(H10) The Laplace exponent $\phi$ of $S$ is a complete Bernstein function.

\vskip 2pt
Assumptions (A1) or (A2) or (A3) imply assumptions (A).
\end{proposition}

More concrete examples of processes satisfying assumptions (A) are in Section 7. 
 
\begin{remark}
Conditions (H5), (H6), (H7), (H8) are exactly the same as conditions (A1), (A2), (A3), (1.2) in a recent, very interesting paper by P. Kim and A. Mimica \cite{KM2012}. Notation used in this paper and in \cite{KM2012} is slightly different.
\end{remark}

Our gradient estimates of harmonic functions for L{\'evy} processes are based on the following observation about a gradient of transition densities for these processes (cf. also Proposition \ref{dertransitionold}).
\begin{theorem}
\label{dertransition}
Let $X$ be a pure-jump isotropic L{\'e}vy process in $\R^d$ with the characteristic exponent $\psi$. We assume that its L{\'e}vy measure is infinite and has the density $\nu(x) = \nu(|x|)$ such that  
$\nu(r) $ is nonincreasing, absolutely continuous and $-\nu'(r)/r$ is nonincreasing.  We denote transition densities of $X$ by $p_t(x) = p_t(|x|)$. Then there exists a L{\'e}vy process $X_t^{(d+2)}$ in $\R^{d+2}$ with the characteristic exponent $\psi^{(d+2)}(\xi) = \psi(|\xi|)$, $\xi \in \R^{d+2}$ and the radial, radially nonincreasing transition density $p_{t}^{(d+2)}(x) = p_{t}^{(d+2)}(|x|)$ satisfying
\begin{equation}
\label{derivativeptr}
p_t^{(d+2)}(r) = \frac{-1}{2 \pi r} \frac{d}{dr} p_t(r), \quad \quad r > 0.
\end{equation}
Moreover  $p_t^{(d+2)}$ is continuous at any $x\neq0$.
\end{theorem}

\begin{remark}
\label{remsubordinated}
Note that if $X_t = B_{S_t}$ is a subordinate Brownian motion and the Levy measure of $S$ is infinite then the above result is obvious and well-known.  We note that  the assumptions
of Theorem \ref{dertransition}  on $\nu(x)$ are automatically satisfied in this case.
\end{remark} 

Let $\varphi$ be a non-negative, non-zero function on $[0,\infty)$. We say that $\varphi$ satisfies {\it{weak upper scaling condition}} WUSC($\overline{\alpha},\theta_0,\overline{C}$) if there are numbers $\overline{\alpha} \in (0,2)$, $\theta_0 \ge 0$ and $\overline{C} > 0$ such that
$$
\varphi(\lambda \theta) \le \overline{C} \lambda^{\overline{\alpha}} \varphi(\theta), \quad \text{for} \quad \lambda \ge 1, \, \theta \ge \theta_0.
$$

Using Theorem \ref{dertransition} and the estimates of $p_t(x)$ obtained in \cite[Corollary 7, Theorem 21]{BGR2013} we obtain the following result which seems to be of independent interest.
\begin{corollary}\label{derestimates}
Let $X$ be an isotropic L{\'e}vy process in $\R^d$ satisfying assumptions $A$. Then its transition density $p_t(x) = p_t(|x|)$ satisfies
$$
\left|\frac{d}{dr} p_t(r)\right| \le c(d) \frac{1\wedge t\psi^*(1/r)}{ r^{d+1}},  \quad t,r>0.
$$
If additionally $\psi$ satisfies WLSC($\underline{\alpha},\theta_0,\underline{C}$), then 
\begin{equation*}
\label{lowergradpt}
\left|\frac{d}{dr} p_t(r)\right| \le c(d,\underline{\alpha})\frac r{\underline{C}^{(d+2)/\underline{\alpha}+1}}\left( [\psi^-(1/t)]^{d+2}\wedge \frac{ t\psi^*(1/r)}{ r^{d+2}}\right),  \quad t\psi^*(\theta_0)\le 1/\pi^2.
\end{equation*}
If additionally $\psi$ satisfies WLSC($\underline{\alpha},\theta_0,\underline{C}$) and WUSC($\overline{\alpha},\theta_0,\overline{C})$, then we have
\begin{equation*}
\label{uppergradpt}
\left|\frac{d}{dr} p_t(r)\right| \ge c^* r\left( [\psi^-(1/t)]^{d+2}\wedge \frac{ t\psi^*(1/r)}{ r^{d+2}}\right),  \quad t\psi^*(\theta_0/r_0)\le 1, \quad r< r_0/\theta_0,
\end{equation*}
where $ c^*=c^*(d, \underline{\alpha},\overline{\alpha},\underline{C},\overline{C}), \quad  r_0=r_0(d, \underline{\alpha},\overline{\alpha},\underline{C},\overline{C})$. Note that if the scaling conditions are global, that is $\theta_0=0$, then the last two estimates hold for all $t, r>0$. Here  $\psi^-$ denotes the generalized inverse of $\psi^*(r)=sup_{\rho\le r}\psi(\rho)$.
\end{corollary}

Theorem \ref{main} implies the following result.

\begin{corollary}
\label{gradientGreenmain}
Let $X$ be an isotropic L{\'e}vy process in $\R^d$ satisfying assumptions (A). Let $D \subset \R^d$ be an open, nonempty set, if $X$ is not transient we assume additionally that $D$ is bounded. Let $G_D(x,y)$ be the Green function corresponding to the process $X$ for the set $D$ (for the definition of $G_D$ see Preliminaries). Then $\nabla_x G_D(x,y)$ exists for any $x, y \in D$, $x \ne y$,  and we have
\begin{equation}
\label{eqGreen}
|\nabla_x G_D(x,y)| \le c\frac{G_D(x,y)}{\delta_D(x) \wedge |x - y| \wedge 1}, \quad \quad x \in D, 
\end{equation} 
where $\delta_D(x) = \dist(x,\partial D)$.
\end{corollary}

The potential theory of L{\'e}vy processes, especially subordinate Brownian motions, has attracted a lot of attention during recent years see e.g. \cite{BKK2013, CKS2010, CKS2013, KimSongVondracek2012, KimSongVondracek2012a}. Our study is in the scope of this type of research. The assertion of Theorem \ref{main} is well known for  harmonic functions with respect to the Brownian motion and symmetric $\alpha$-stable proceses \cite{BKN2002}, where explicit formulas for the Poisson kernel for a ball served as a main tool. For the processes treated in the present  paper such formulas are not available and we had to take another approach based on Theorem \ref{dertransition}.   A result similar to Theorem \ref{main} is also known for harmonic functions with respect to Schr{\"o}dinger operators based on the Laplacian and the fractional Laplacian (\cite{CZ1990}, \cite{BKN2002}, \cite{K2013}). Some probabilistic ideas used in our paper are, to some extent, similar to the concept of coupling from M. Cranston and Z. Zhao's paper \cite{CZ1990}. The idea of coupling for L{\'e}vy processes was used by R. Schilling, P. Sztonyk and J. Wang in \cite{SSzW2012} where gradient estimates of the corresponding transition semigroups were derived. They are of the type $||\nabla P_t u||_{\infty} \le c ||u||_{\infty} f^{-1}(1/t)$, $u \in \mathcal{B}_b(\R^d)$, where $|\Re \psi(\xi)| \approx f(|\xi|)$ as $|\xi| \to \infty$. Recently, K. Kaleta and P. Sztonyk in \cite{KSz2013} obtained also gradient estimates of transition densities for L{\'e}vy processes. Note that our sharp, two-sided estimates (\ref{lowergradpt}-\ref{uppergradpt}) obtained in Corollary \ref{derestimates} are of different form than those obtained in \cite{SSzW2012} and \cite{KSz2013}. The results in \cite{SSzW2012} and \cite{KSz2013} are obtained under more general assumptions but they are not as sharp as ours.

Estimates of derivatives of harmonic functions with respect to some (not necessarily symmetric) $\alpha$-stable processes were obtained by P. Sztonyk in \cite{Sz2010}. However these estimates were obtained for a different class of processes and they are not pointwise (as our estimates) which is crucial in applications (see below). Moreover, P. Sztonyk in \cite{Sz2010} obtained gradient estimates only for $\alpha \ge 1$.

It seems that for applications the most important are Corollaries \ref{derestimates} and \ref{gradientGreenmain}. For example they could be used to study operators $L + b \nabla$ (where $L$ is the generator of a L{\'e}vy process in $\R^d$ and $b: \R^d \to \R$). Such operators (when $L = -(-\Delta)^{\alpha/2}$) are intensively studied both in the theory of partial differential equations (see e.g. \cite{S2013}) and in the theory of stochastic processes (see e.g. \cite{BJ2007}, \cite{CKS2012}). Especially the techniques used by K. Bogdan, T. Jakubowski in \cite{BJ2007} and by Z.-Q. Chen, P. Kim, R. Song in \cite{CKS2012} demand pointwise gradient estimates of transition densities and a Green function exactly of the type presented in our paper. It seems that our estimates would allow to extend results from \cite{BJ2007} and \cite{CKS2012} (obtained there for the fractional Laplacian) to more general generators of L{\'e}vy processes.  Our estimates also seem  to be useful in the study of the spectral theory related to L{\'e}vy processes and Schr{\"o}dinger operators based on their generators.

The paper is organized as follows. Section 2 is preliminary. In Section 3 we prove Theorem \ref{dertransition}. Section 4 concerns the difference process. In Section 5 we prove some auxiliary facts concerning the Green function and the L{\'e}vy measure. In Section 6 we prove Theorem \ref{main}. In the last section we present examples of processes satisfying assumptions (A) and we also present an example of a harmonic function with respect to some pure-jump, isotropic L{\'e}vy process for which the gradient does not exist at some point.

\section{Preliminaries}

For $x \in \R^d$ and $r > 0$ we let $B(x,r) = \{y \in \R^d: \, |y - x| < r\}$. By $a \wedge b$ we denote $\min(a,b)$ for $a, b \in \R$. When $D \subset \R^d$ is an open set we denote by $\mathcal{B}(D)$ a family of Borel subsets of $D$.

A Borel measure on $\R^d$ is called {\it{isotropic unimodal}} if on $\R^d \setminus \{0\}$ it is absolutely continuous with respect to the Lebesgue measure and has a finite radial nonincreasing density function (such measures may have an atom at the origin). 

A L{\'e}vy process $X = (X_t, t \ge 0)$ in $\R^d$ is called isotropic unimodal if its transition probability $p_t(dx)$ is isotropic unimodal for all $t > 0$.  When additionally $X$ is a pure-jump process then the following L{\'e}vy-Khintchine formula holds for $t > 0$ and $\xi \in \R^d$,
$$
E^0 e^{i\xi X_t} = \int_{\R^d} e^{i\xi x} p_t(dx) = e^{-t \psi(\xi)} \quad \text{where} \quad \psi(\xi) = \int_{\R^d} (1 - \cos(\xi x)) \nu(dx).
$$
$\psi$ is the characteristic exponent of $X$ and $\nu$ is the L{\'e}vy measure of $X$. $E^0$ is the expected value for the process $X$ starting from $0$. Recall that a L{\'e}vy measure is a measure concentrated on $\R^d \setminus \{0\}$ such that $\int_{\R^d} (|x|^2 \wedge 1) \nu(dx) < \infty$. Isotropic unimodal pure-jump L{\'e}vy measures are characterized in \cite{W1983} by unimodal L{\'e}vy measures $\nu(dx) = \nu(x) \, dx = \nu(|x|) \, dx$. 

Unless explicitly stated otherwise in what follows we assume that $X$ is a pure-jump isotropic unimodal L{\'e}vy process in $\R^d$ with (isotropic unimodal) infinite L{\'e}vy measure $\nu$. Then for any $t > 0$ the measure $p_t(dx)$ has a radial, radially nonincreasing density function $p_t(x) = p_t(|x|)$ on $\R^d$ with no atom at the origin. However, it may happen that $p_t(0) = \infty$, for some $t > 0$. As usual, we denote by $P^x$ and $E^x$ the probability measure and the corresponding expectation for the the process starting from $x \in \R^d$.

The process $X$ is said to be transient if $P^0(\lim_{t \to \infty} |X_t| = \infty) = 1$. For $ d \ge 3$ the process $X$ is always transient (see e.g. \cite{G2013}, the remark after Lemma 5).

For a transient process by $U$ we denote the potential kernel for the process $X$. That is 
$$U(x)= \int_0^\infty p_{t}(x)\,dt,\quad x \in \R^d.$$
 By $U^{(d+2)}$ we denote the potential kernel for the process $X^{(d+2)}$ defined in Theorem \ref{dertransition}. Since the process $X^{(d+2)}$ lives in at least three-dimensional space then $U^{(d+2)}(x)<\infty, \, x\neq 0$.  T. Grzywny in \cite{G2013} obtained estimates of the potential kernel in terms of the symbol $\psi$, which play an important role in the present paper.

We define the maximal characteristic function $\psi^*(r) = \sup_{s \le r} \psi(s)$, where $r \ge 0$. We have \cite[Proposition 2]{BGR2013} $\psi(r) \le \psi^*(r) \le \pi^2 \psi(r)$, $r \ge 0$. The function  $\psi^*$ has the property  \cite[Lemma 1]{G2013}, 
$$\psi^*(r)\le 2\frac{1+s^2}{s^2}\psi^*(sr),\quad r,s>0.$$
In the sequel the following  nondecreasing function will play an important role in our development 
$$
\V(r) = \left(\psi^*\left(\frac{1}{r}\right)\right)^{-1/2}, \quad  r > 0, 
$$
and $\V(0) = 0$.  As an immediate consequence of the above property of $\psi^*$  we have 
\begin{equation}\V(sr)\le  \sqrt{2(1+s^2)}\V(r), \quad r,s>0. \label{psi_star}\end{equation}

This property will be frequently used throughout the paper without further mention while comparing values of  $\V$ at points with fixed ratio.  
There are many important quantitie related to the process $X$, which enjoy precise estimates in terms of  $\V(r)$.
We have \cite[Corollary 7]{BGR2013},
\begin{equation}
\label{denestimates}
p_t(x) \le \frac{c t}{\V^2({|x|}) |x|^d}, \quad t > 0, \, x \in \R^d,
\end{equation}
\begin{equation}
\label{Levymeasureestimates}
\nu(x) \le \frac{c}{\V^2({|x|}) |x|^d}, \quad x \in \R^d,
\end{equation}
where $c = c(d)$. Under some further conditions (\ref{denestimates}-\ref{Levymeasureestimates}) can be reveresed  (\cite{BGR2013}). 
Note that the upper bound of the L\'evy density yields 
\begin{equation}\label{0_Levy}\limsup_{r\searrow 0} r^{d+2}\nu(r)\le \limsup_{r\searrow 0}  \frac{c r^2}{\V^2(r)}=0, \end{equation}
and
\begin{equation}\label{infty_Levy}\limsup_{r\to \infty} r^{d}\nu(r)\le \limsup_{r\to \infty}  \frac{c}{\V^2(r)}=0.\end{equation}
For the proof that $\limsup_{r\searrow 0}  \frac{r^2}{\V^2(r)}=0$, see \cite[Lemma 2.5]{BGR2013_1}.

The first exit time of an open, nonempty set $D \subset \R^d$ of the process $X$ is defined by $\tau_D = \inf\{t > 0: \, X_t \notin D\}$.

\begin{definition}
A Borel function $f: \R^d \to \R$ is called {\it{harmonic}} with respect to the process $X$ in an open, nonempty set $D \subset \R^d$ if for any bounded, open, nonempty set $B$, such that $\overline{B} \subset D$
$$
f(x) = E^x\left(f\left(X(\tau_B)\right)\right), \quad \quad x \in B.
$$
We understand that the expectation is absolutely convergent.
\end{definition}

\begin{definition}\label{Harnack}
The {\it{scale invariant Harnack inequality}} holds for the process $X$ if there exists a constant $a_2$ such that for any $x_0 \in R^d$, $r \in (0,1]$, and any function $h$ nonnegative on $\R^d$ and harmonic in a ball $B(x_0,r)$, 
$$
\sup_{x \in B(x_0,r/2)} h(x) \le a_2 \inf_{x \in B(x_0,r/2)} h(x).
$$
\end{definition}

Let $D \subset \R^d$ be an open, nonempty set. If $X$ is not transient we assume additionally that $D$ is bounded. We define a {\it{killed process}} $X_t^D$ by $X_t^D = X_t$ if $t < \tau_D$ and $X_t^D = \partial$ otherwise, where $\partial$ is some point adjoined to $D$ (usually called cemetary). The transition density for $X_t^D$ on $D$ is given by
$$
p_D(t,x,y) = p_t(x-y) - E^x(p_{t - \tau_D}(X(\tau_D),y), \, t > \tau_D), \quad x,y \in D, \, t > 0,
$$ 
that is for any Borel set $A \subset \R^d$ we have
$$
P^x(X_t^D \in A) = \int_A p_D(t,x,y) \, dy, \quad x \in D, \, t> 0.
$$
We have $p_D(t,x,y) = p_D(t,y,x)$, $x, y \in D$, $t > 0$.
We define the {\it{Green function}} for $X_t^D$ by 
$$
G_D(x,y) = \int_0^{\infty} p_D(t,x,y) \, dt, \quad x,y \in D,
$$
$G_D(x,y) = 0$ if $x \notin D$ or $y \notin D$. For any Borel set $A \subset \R^d$ we have
$$
E^x \int_0^{\tau_D} 1_A(X_t) \, dt = \int_A G_D(x,y) \, dy, \quad x \in D.
$$
In particular if we set $A=D$ we obtain 
$$
E^x \tau_D  = \int_D G_D(x,y) \, dy, \quad x \in D.
$$
We have $G_D(x,y) = G_D(y,x)$, $x, y \in D$. For a fixed $y \in D$ the function $x \to G_D(x,y)$ is harmonic with respect to $X$ in $D \setminus \{y\}$.  
The estimates of $E^x \tau_D$ when $D$ is a ball play an important role in the paper. Here we record very useful  upper and lower  bounds  in terms of the function $\V$ (see e.g. \cite[Lemmas 2.3, 2.7]{BGR2013_1}, see also \cite{S1998}).
\begin{lemma}\label{Exit_ball} There is an absolute constant $C_1$, and   a constant  $C_2=C_2(d)$ such that for any $r > 0$ we have
$$
E^x \tau_{B(0,r)}  \le C_1 L(r)L(\delta(x)),  \quad x \in B(0,r)
$$
and 
$$
E^x \tau_{B(0,r)}  \ge C_2 L^2(r),  \quad x \in B(0,r/2),
$$
where $\delta(x)=\delta_{B(0,r)}(x)$.
\end{lemma}

Let $D \subset \R^d$ be a bounded, open, nonempty set. The distribution $P^x(X(\tau_D) \in \cdot)$ is called the {\it{harmonic measure}} with respect to $X$. The harmonic measure for Borel sets $A \subset (\overline{D})^c$ is given by the Ikeda-Watanabe formula \cite{IW1962},
\begin{equation}
\label{IW}
P^x(X(\tau_D) \in A) = \int_A \int_D G_D(x,y) \nu(y-z) \, dy \, dz, \quad x \in D.
\end{equation} 

When $D \subset \R^d$ is a bounded, open Lipschitz set then we have  \cite{Sztonyk2000}, \cite{Millar1975},
\begin{equation}
\label{Xboundary}
P^x(X(\tau_D) \in \partial D) = 0, \quad x \in D.
\end{equation}
It follows that for such sets $D$ the Ikeda-Watanabe formula (\ref{IW}) holds for any Borel set $A \subset D^c$.
Let $D \subset \R^d$ be a bounded, open, nonempty set. For any $s > 0$, $x \in D$, $z \in (\overline{D})^c$ put 
\begin{equation}
\label{hformula}
h_D(x,s,z) = \int_D p_D(s,x,y) \nu(y-z) \, dy.
\end{equation}
By Ikeda-Watanabe formula \cite{IW1962} for any Borel $A \subset (0,\infty)$, $B \subset D^c$ we have 
\begin{equation}
\label{IW2}
P^x(\tau_D \in A, X(\tau_D) \in B) = \int_A \int_B h_D(x,s,z) \, dz \, ds, \quad x \in D.
\end{equation}

From \cite[Lemma 2.1]{BGR2013_1} we have the following estimate.
\begin{lemma}
\label{UBHP} Let $z \in \R^d$, $s > 0$, $D \subset B(z,s)$ be a bounded, open, nonempty set and $y \in D \cap B(z,s/2)$. There is a constant $c=c(d)$ such that 
$$
P^y(X(\tau_D) \in B^c(z,s)) \le c \frac{E^y(\tau_D)}{\V^2(s)}.
$$
\end{lemma}

Important examples of isotropic unimodal L{\'e}vy processes are subordinate Brownian motions. By $S = (S_t, t\ge 0)$ we denote a {\it{subordinator}} i.e. a nondecreasing L{\'e}vy process starting from $0$. The Laplace transform of $S$ is of the form 
$$
E e^{-\lambda S_t} = e^{- t \phi(\lambda)}, \quad \lambda \ge 0, \, t \ge 0,
$$
where $\phi$ is called the Laplace exponent of $S$. $\phi$ is a Bernstein function and has the following representation 
\begin{equation}
\label{phirep}
\phi(\lambda) = b \lambda + \int_{(0,\infty)} (1 - e^{-\lambda u }) \, \nu_S(du)
\end{equation}
where $b \ge 0$ and $\nu_S$ is a L{\'e}vy measure on $(0, \infty)$ such that $\int_{(0,\infty)} (1 \wedge u) \, \nu_S(du) < \infty$.  

Let $B = (B_t, t \ge 0)$ be a Brownian motion in $\R^d$ (with a generator $\Delta$) and let $S$ be an independent subordinator. We define a new process $X_t = B_{S_t}$ and call it a {\it{subordinate Brownian motion}}. Let us assume that $b = 0$ and $\nu_S(0,\infty) = \infty$ in (\ref{phirep}).  This process is a L{\'e}vy process with the characteristic exponent $\psi(\xi) = \phi(|\xi|^2)$. Moreover $X$ has the L{\'e}vy measure $\nu(dx) = \nu(x) \, dx = \nu(|x|) \, dx$ given by \cite[Theorem 30.1]{Sato1999}
$$
\nu(r) = \int_{(0,\infty)} (4 \pi t)^{-d/2} \exp\left(-\frac{r^2}{4t}\right) \, \nu_S(dt), \quad r > 0.
$$

The next lemma seems to be  known but we could not find any reference so we decided to present its short proof. 
\begin{lemma}
\label{continuity}
Let $X$ be a pure-jump isotropic L{\'e}vy process in $\R^d$. We assume that its L{\'e}vy measure is infinite and has the density $\nu(x) = \nu(|x|)$ which is radially nonincreasing. Then for each $t > 0$ the density function $p_t(x)$ of the process is continuous on $\R^d \setminus \{0\}$.
\end{lemma}
\begin{proof} Let $p_t$ be the distribution of $X_t$.
It is well known that under above assumptions for each $t > 0$ the measure $p_t$ has a radial, nonincreasing density function $p_t(x)$ on $\R^d \setminus \{0\}$ and $p_t$ has no atom at $\{0\}$. 

Let us denote $f_t(x) = \int_{\R^d} p_{t/2}(x-y) p_{t/2}(y) \, dy$, $t > 0$, $x \in \R^d$. We have $p_t(x) = f_t(x)$ a.s. so it is enough to show that for each $t > 0$ the function $f_t$ is continuous on $\R^d \setminus \{0\}$. 

Fix $t > 0$, $z \in \R^d$, $z \ne 0$ and $\ve > 0$. Let $M = \sup_{y \in B^c(0,|z|/2)} p_{t/2}(y)$. Take $\delta \in (0,|z|/4)$ such that $\int_{B(0,\delta)} p_{t/2}(y) \, dy < \ve/(4M)$. For any $x \in B(z,|z|/4)$ we have
\begin{equation}
\label{smallball}
\int_{B(0,\delta)} p_{t/2}(x-y) p_{t/2}(y) \, dy \le M \int_{B(0,\delta)} p_{t/2}(y) \, dy < \ve/4.
\end{equation}
Denote $f^{(1)}(x) = \int_{B(0,\delta)} p_{t/2}(x-y) p_{t/2}(y) \, dy$ and $f^{(2)}(x) = \int_{B^c(0,\delta)} p_{t/2}(x-y) p_{t/2}(y) \, dy$. (\ref{smallball}) implies that for $x \in B(z,|z|/4)$ we have $|f^{(1)}(x) - f^{(1)}(z)| \le \ve/2$. On the other hand note that $f^{(2)}(x)$ is the convolution of the function $p_{t/2}(y) \in L^1(\R^d)$ and the bounded function $1_{[\delta,\infty)}(|y|)p_{t/2}(y)$. Hence $f^{(2)}(x)$ is continuous on $\R^d$. It follows that $f_t(x) = f^{(1)}(x) + f^{(2)}(x)$ is continuous at $x = z$.
\end{proof}

\begin{lemma}
\label{contunimodal}
For any $\ve \in (0,1]$ let $f_{\ve} \in L^1(\R^d)$ and let $f \in L^1(\R^d)$. Assume that all $f_{\ve}$, $f$ are nonnegative, continuous, radial, radially nonincreasing and $f_{\ve} \to f$ weakly as $ \ve \to 0$ (as measures on $\R^d$). Then the convergence is  pointwise at any $x \neq 0$.
\end{lemma}
\begin{proof} Let $0<a<b<\infty$. From the weak convergence 
$$\lim_{\ve \to 0}\int_a^{b} f_{\ve}(r) r^{d-1} \, dr =\int_a^{b} f(r) r^{d-1} \, dr. $$
By monotonicity  $\int_a^{b} f(r) r^{d-1} \, dr \le f(a) b^{d-1}(b-a) $ and 
$ f_{\ve}(b) a^{d-1}(b-a) \le \int_a^{b} f_{\ve}(r) r^{d-1} \, dr$. It follows that
$$\limsup_{\ve \to 0}  f_{\ve}(b) a^{d-1}\le f(a) b^{d-1}.$$
Using continuity of $f$ and passing $a\nearrow b$ we obtain 
$\limsup_{\ve \to 0}  f_{\ve}(b) \le f(b)$. By a symmetric argument we have 
$\liminf_{\ve \to 0}  f_{\ve}(a) \ge f(a)$.
\end{proof}

Now we will show Proposition \ref{classes}. 

\begin{proof}[Proof of Proposition \ref{classes}]

First, we show that assumptions (A1) imply (A).  If $d \ge 3$ then \cite[Theorem 1]{G2013} gives (H2). If $d \le 2$ then Theorem \ref{dertransition} and \cite[Corollary 6]{G2013} gives (H2).
 
In the next step we prove that assumptions (A2) imply (A). Recall that conditions (H5), (H6), (H7), (H8) are exactly the same as conditions (A1), (A2), (A3), (1.2) in \cite{KM2012}. (H2) follows from \cite[Theorem 1.2]{KM2012}. Remark \ref{remsubordinated} implies that $\nu(r)$ is nonincreasing, absolutely continuous and $-\nu'(r)/r$ is nonincreasing. Now we will show that $\nu(r) \le a \nu(2r), r \in (0,1]$ for some constant $a$.  It is clear that the scaling property for  $\phi' $ (H7) implies that  $\phi'( r^{-2})\le a' \phi'( (2r)^{-2}), \,  r\in (0,1/\sqrt{4\theta_0}]$. Since  $\nu(r) \approx \phi'(r^{-2}) r^{-d-2}, r \in (0,1]$ (see \cite[Proposition 4.2]{KM2012})) we obtain there is a constant $a''$ such that $\nu(r)\le a'' \nu(2r), r\in (0,1/\sqrt{4\theta_0}]$. Clearly this inequality holds for all $r \in (0,1]$ with (possibly) a different constant. 

Finally, we justify that assumptions (A3) imply (A). This again follows from arguments presented in the paper by P. Kim and A. Mimica \cite{KM2012}. (H4) and (H10) imply (H5) and (H6). (H4) and (H10) imply also (H8), see \cite[Remark 4.3]{KM2012}. So (H4), (H5), (H6), (H7), (H8) hold. Hence we can use \cite[Theorem 1.2]{KM2012} and get (H2). Remark \ref{remsubordinated} implies that $\nu(r)$ is nonincreasing, absolutely continuous and $-\nu'(r)/r$ is nonincreasing. The fact that $\nu(r) \le a \nu(2r), r \in (0,1]$ for some constant $a$ can be shown in the same way as in case (A2).
\end{proof}

\begin{remark}
\label{constants}
All constants appearing in this paper are positive and finite. We write $\kappa = \kappa(a,\ldots,z)$ to emphasize that $\kappa$ depends only on $a,\ldots,z$. We adopt the convention that constants denoted by $c$ (or $c_1$, $c_2$) may change their value from one use to the next.
In the whole paper, unless is explicitly stated otherwise, we understand that constants denoted by $c$ (or $c_1$, $c_2$) depend on $d, a_1, a_2$, where $a_1, a_2$ appear in (H1) and 
Definition \ref{Harnack}, respectively.
In particular, it  applies to  the constant $c$ in (\ref{maineq}). 
\end{remark}

\section{The derivative of the transition density}

We denote the Fourier transform of $f \in L^1(\R^d)$ by $\calF f(y) = \int_{\R^d} e^{-ixy} f(x) \, dx$, $y \in \R^d$ and the inverse Fourier transform of $f \in L^1(\R^d)$ by $\tcalF f(y) = (2 \pi)^{-d} \int_{\R^d} e^{ixy} f(x) \, dx$, $y \in \R^d$. It is well known that for any real, radial $f \in L^1(\R^d)$ we have $\calF f(y) = \calFd f(|y|)$, $y \in \R^d$, $y \ne 0$, where
$$
\calFd f(R) = (2\pi)^{d/2} \int_0^{\infty} \frac{J_{\frac{d-2}{2}}(rR)}{(rR)^{\frac{d-2}{2}}} f(r) r^{d-1} \, dr, \quad R>0.
$$
Here $J_{\alpha}$ is the Bessel function of order $\alpha$. Similarly for any real, radial $f \in L^1(\R^d)$ we have $\tcalF f(y) = \tcalFd f(|y|)$, $y \in \R^d$, $y \ne 0$, where $\tcalFd f(R) = (2 \pi)^{-d} \calFd f(R)$, $R > 0$.

We will use the following result from \cite{GT2013}. Let $f:[0,\infty) \to \R$ be a Borel function satisfying  $\int_0^{\infty} |f(r)| (r^{d-1} + r^{d+1}) dr < \infty$. Then we have
\begin{equation}
\label{dimensiond2}
\frac{d}{dR} (\tcalFd f)(R) = -2 \pi R \tcalFdd f(R), \quad R>0.
\end{equation}

We first prove Proposition \ref{dertransitionold} which is a version of Theorem \ref{dertransition} with slightly changed assumptions. We will use this proposition in the proof of Theorem \ref{dertransition} but it seems that Proposition \ref{dertransitionold} is of independent interest.
\begin{proposition}
\label{dertransitionold}
 Let $X$ be a pure-jump isotropic L{\'e}vy process in $\R^d$ with the characteristic exponent $\psi$ and the transition density $p_t(x) = p_t(|x|)$. We assume that its L{\'e}vy measure  has the density $\nu(x) = \nu(|x|)$. We further assume that $\psi$ satisfies $\lim_{\rho \to \infty} ({\psi(\rho)}/{\log(\rho)}) = \infty$ and $\nu(r)$ is nonincreasing.  Then there exists a L{\'e}vy process $X^{(d+2)}$ in $\R^{d+2}$ with the characteristic exponent $\psi^{(d+2)}(\xi) = \psi(|\xi|)$, $\xi \in \R^{d+2}$ and the transition density $p_{t}^{(d+2)}(x) = p_{t}^{(d+2)}(|x|)$ satisfying
$$
p_t^{(d+2)}(r) = \frac{-1}{2 \pi r} \frac{d}{dr} p_t(r), \quad \quad r > 0.
$$
\end{proposition}
\begin{proof}
Put $s_t(r):=e^{-t\psi(r)}$, $r \ge 0$. By the fact that $\psi$ satisfies $\lim_{\rho \to \infty} ({\psi(\rho)}/{\log(\rho)}) = \infty$ we obtain that $\int_0^{\infty} |s_t(r)| (r^{d-1} + r^{d+1}) dr < \infty$, for any $t > 0$. We have $p_t(x) = (2 \pi)^{-d} \int_{\R^d} e^{ixy} e^{-t \psi(y)} \, dy$, $x \in \R^d$, so $p_t(R) = \tcalFd s_t(R)$, $R > 0$.

Now let us define
$$
p_t^{(d+2)}(x)=(2 \pi)^{-d-2} \int_{\R^{d+2}} e^{ixy} e^{-t \psi(y)} \, dy, \quad x \in \R^{d+2}, \, t > 0.
$$
By (\ref{dimensiond2}) we have
\begin{equation*}
p_t^{(d+2)}(R) = \tcalFdd s_t(R) =  \frac{-1}{2 \pi R} \frac{d}{dR}(\tcalFd s_t)(R)
= \frac{-1}{2 \pi R} \frac{d}{dR} p_t(R).
\end{equation*}
Note that $p_t(R)$ is nonincreasing so $p_t^{(d+2)}(R) \ge 0$, $R > 0$.

It follows that $\calF p_t^{(d+2)}(x) = \calF \tcalF(e^{-t\psi(\cdot)})(x) = e^{-t\psi(x)}$, $x \in \R^{d+2}$, $t > 0$. Since $\psi(0)=0$ we observe that $ p_t^{(d+2)}(x)$ is a probability density. Moreover, this implies that $p_t^{(d+2)} \ast p_s^{(d+2)} = p_{t+s}^{(d+2)}$, $s,t > 0$ and $p_t^{(d+2)}$ tends weakly to $\delta_0$ as $t \to 0$, where $\delta_0$ is the Dirac delta at $0$.

In consequence there exists a L{\'e}vy process $\{X_t^{(d+2)}\}_{t \ge 0}$ in $\R^{d+2}$ with the transition density $p_t^{(d+2)}(x)$ and the L{\'e}vy-Khinchin exponent $\psi$. 

\end{proof}

\begin{proof}[proof of Theorem \ref{dertransition}]

Let us define
$$
\nu^{(d+2)}(R) = -\frac{1}{2 \pi R} \frac{d \nu}{dR}(R), \quad R > 0,
$$
and $\nu^{(d+2)}(x) = \nu^{(d+2)}(|x|)$, $x \in \R^{d+2}$, $x \ne 0$. Let $\nu^{(d+2)}$ be the measure on $\R^{d+2}$ given by $\nu^{(d+2)}(\{0\}) = 0$ and $\nu^{(d+2)}(dx) = \nu^{(d+2)}(x)dx$, $x \in \R^{d+2}$, $x \ne 0$.

Now we will show that $\int_{R^{d+2}} (1 \wedge |x|^2) \, \nu^{(d+2)}(dx) < \infty$.
Clearly, $ (-\nu)'(|x|) (2\pi|x|)^{-1} \ge 0$. Note that we have to show $\int_{\R^{d+2}} (-\nu)'(|x|) (2\pi|x|)^{-1} (1 \wedge |x|^2) \, dx < \infty$. It is enough to prove $\int_{0}^1  (-\nu)'(r) r^{d+2} \, dr < \infty$ and $\int_{1}^{\infty} (-\nu)'(r) r^{d} \, dr < \infty$.  Integrating by parts we get
$$
\int_{\varepsilon}^1 (-\nu)'(r) r^{d+2} \, dr =  - \nu(r) r^{d+2}|_{\varepsilon}^1 + (d+2) \int_{\varepsilon}^1 \nu(r) r^{d+1} \, dr,
$$
where $\varepsilon \in (0,1)$ is arbitrary. Since $\int_{0}^1 \nu(r) r^{d+1} \, dr < \infty$ we must have $\liminf_{r \to 0} \nu(r) r^{d+2} = 0$. It follows that $\int_{0}^1  (-\nu)'(r) r^{d+2} \, dr < \infty$. Again integrating by parts we obtain
\begin{equation*}
\int_{1}^{N}  (-\nu)'(r) r^{d} \, dr = - \nu(r) r^{d}|_{1}^N + d \int_{1}^N \nu(r) r^{d-1} \, dr,
\end{equation*}
where $N > 1$ is arbitrary. It follows that $\int_{1}^{\infty} (-\nu)'(r) r^{d} \, dr < \infty$.

Hence the measure $\nu^{(d+2)}$ satisfies the conditions of a L{\'e}vy measure  in $\R^{d+2}$. Let $X^{(d+2)} = (X^{(d+2)}_t, t \ge 0)$ be a pure-jump Levy process in $\R^{d+2}$ with a L{\'e}vy measure $\nu^{(d+2)}$. One can easily check that $\nu^{(d+2)}(\R^{d+2}) = \infty$. Indeed,
\begin{equation*}
\int_{\epsilon}^{1}  (-\nu)'(r) r^{d} \, dr = - \nu(r) r^{d}|_{\epsilon}^1 + d \int_{\epsilon}^{1} \nu(r) r^{d-1} \, dr\to \infty, \quad \epsilon\searrow 0. 
\end{equation*}

 Note that the Levy measure $\nu^{(d+2)}$ has the density which is radial and radially nonincreasing. Let $p^{(d+2)}_t$ be the distribution of $X^{(d+2)}_t$.
It follows that for each $t > 0$ the measure $p^{(d+2)}_t$ has a radial, nonincreasing bounded density function $p^{(d+2)}_t(x) = p^{(d+2)}_t(|x|)$ on $\R^{d+2}$. 

Let $\psi^{(d+2)}$ be the characteristic exponent for the process $X^{(d+2)}$. Now our aim is to show that $\psi^{(d+2)}(R) = \psi(R)$, $R > 0$. We have 
$$
\psi^{(d+2)}(\xi) = \int_{\R^{d+2}} (1 - \cos(\xi x)) \, \nu^{(d+2)}(x) \, dx, \quad \xi \in \R^{d+2}.
$$

So to prove $\psi^{(d+2)}(R) = \psi(R)$, $R > 0$ it is enough to show that 
$$
\psi(\xi) =  \int_{\R^{d+2}} (1- \cos(\xi x)) \left(-\frac{1}{2 \pi |x|} (\nu)'(|x|)\right) \, dx, \quad \xi \in \R^{d+2}.
$$
Hence it is sufficient to prove for $R >0$,
$$
\psi(R) = \int_0^{\infty} \left(\omega_{d+1} - (2\pi)^{\frac{d+2}{2}}(rR)^{-\frac{d}{2}} J_{\frac{d}{2}}(rR) \right)  \left(-\frac{1}{2 \pi r} \frac{d \nu}{dr}(r)\right) r^{d+1} \, dr,
$$
where $\omega_{d} = 2 \pi^{(d+1)/2}/\Gamma((d+1)/2)$.
Since $\nu$ is the density of the L{\'e}vy measure of $X$ in $\R^d$ we have 
\begin{eqnarray*}
\psi(R) &=& \int_0^{\infty} \left(\omega_{d-1} - (2\pi)^{\frac{d}{2}} (rR)^{-(\frac{d-2}{2})} J_{\frac{d-2}{2}}(rR) \right)  \nu(r) r^{d-1} \, dr\\
&=& \int_0^{\infty} \left(\omega_{d-1} r^{d-1} - (2\pi)^{\frac{d}{2}} (rR)^{\frac{d}{2}} J_{\frac{d-2}{2}}(rR) R^{1-d}  \right)  \nu(r) \, dr.
\end{eqnarray*}
Using the property of Bessel functions $\frac{d}{ds}(s^{\alpha}J_{\alpha}(s)) = s^{\alpha}J_{\alpha-1}(s)$ ($\alpha \in (-1/2,\infty)$, $s > 0$) this is equal to 
$$
I=\int_0^{\infty} \frac{d}{dr}\left(\frac{\omega_{d-1}}{d} r^{d} - (2\pi)^{\frac{d}{2}} (rR)^{\frac{d}{2}} J_{\frac{d}{2}}(rR) R^{-d} \right)  \nu(r) \, dr.
$$
By asymptotics of the Bessel function  $J_{\frac{d}{2}}(r)$  at zero we show that 
$$\frac{\omega_{d-1}}{d} r^{d} - (2\pi)^{\frac{d}{2}} (rR)^{\frac{d}{2}} J_{\frac{d}{2}}(rR) R^{-d}\approx C r^{d+2} R^{2+d/2}. $$ Hence, applying (\ref{0_Levy}),
$$\lim_{r\to 0}\left(\frac{\omega_{d-1}}{d} r^{d} - (2\pi)^{\frac{d}{2}} (rR)^{\frac{d}{2}} J_{\frac{d}{2}}(rR) R^{-d}\right)\nu(r)= \lim_{r\to 0} r^{d+2}\nu(r)=0.$$
Using the fact   the Bessel function  $J_{\frac{d}{2}}(r)$ is bounded  at $\infty$,  we show, applying (\ref{infty_Levy}), that 
$$\lim_{r\to \infty}\left|\frac{\omega_{d-1}}{d} r^{d} - (2\pi)^{\frac{d}{2}} (rR)^{\frac{d}{2}} J_{\frac{d}{2}}(rR) R^{-d}\right|\nu(r)= \lim_{r\to \infty} r^{d}\nu(r)=0.
$$
This justifies that by integrating by parts we obtain 
$$
I=\int_0^{\infty} \left(\omega_{d+1} - (2\pi)^{\frac{d+2}{2}}(rR)^{-\frac{d}{2}} J_{\frac{d}{2}}(rR) \right)  \left(-\frac{1}{2 \pi r} \frac{d \nu}{dr}(r)\right) r^{d+1} \, dr.
$$
So we have finally shown that $X^{(d+2)}$ has the characteristic exponent $\psi$.

Our next aim is to show (\ref{derivativeptr}). 
For any $\ve \in (0,1]$ let $X_{(\ve)} = (X_{(\ve),t}, t \ge 0)$ be the L{\'e}vy process in $\R^d$ with the characteristic exponent $\psi_{\ve}(\xi) = \psi(\xi) + \ve |\xi|$, $\xi \in \R^d$. Let $p_{\ve,t}$ be the distribution of $X_{(\ve),t}$. It follows that for each $t > 0$ the measure $p_{\ve,t}$ has a radial, nonincreasing bounded density function $p_{\ve,t}(x) = p_{\ve,t}(|x|)$ on $\R^{d}$. For any $t > 0$ clearly, $p_{\ve,t} \to p_t$ weakly as $\ve \searrow 0$. All densities $p_{\ve,t}(x)$, $p_t(x)$ are continuous on $\R^{d} \setminus \{0\}$. Hence by Lemma \ref{contunimodal} for any $t > 0$, $x \in \R^d$, $x \ne 0$ we have $p_{\ve,t}(x) \to p_t(x)$  as $\ve \searrow 0$.

By Proposition \ref{dertransitionold} there exists a L{\'e}vy process $X_{(\ve)}^{(d+2)} = (X_{(\ve),t}^{(d+2)}, t \ge 0)$ in $\R^{d+2}$ with the characteristic exponent $\psi_{\ve}^{(d+2)}(\xi) = \psi(|\xi|) + \ve |\xi|$, $\xi \in \R^{d+2}$. Let $p_{\ve,t}^{(d+2)}$ be the distribution of $X_{(\ve),t}^{(d+2)}$. It follows that for each $t > 0$ the measure $p_{\ve,t}^{(d+2)}$ has a radial, nonincreasing bounded density function $p_{\ve,t}^{(d+2)}(x) = p_{\ve,t}^{(d+2)}(|x|)$ on $\R^{d+2}$. For any $t > 0$ clearly, $p_{\ve,t}^{(d+2)} \to p_t^{(d+2)}$ weakly as $\ve \searrow 0$. All densities $p_{\ve,t}^{(d+2)}(x)$, $p_t^{(d+2)}(x)$ are continuous on $\R^{d+2} \setminus \{0\}$. 

Fix $0 < r_1 < r_2 < \infty$. By Proposition \ref{dertransitionold} we have 
$$
p_{\ve,t}(r_2) - p_{\ve,t}(r_1) = \int_{r_1}^{r_2} \frac{\partial}{\partial r} p_{\ve,t}(r) \, dr = 
\int_{r_1}^{r_2} \frac{-1}{2 \pi r} p_{\ve,t}^{(d+2)}(r) \, dr.
$$
Since for any $r > 0$ we have $p_{\ve,t}(r) \to p_t(r)$ as $\ve \searrow 0$ and $p_{\ve,t}^{(d+2)} \to p_t^{(d+2)}$ weakly  as $\ve \searrow 0$ we obtain
$$
p_{t}(r_2) - p_{t}(r_1) = 
\int_{r_1}^{r_2} \frac{-1}{2 \pi r} p_{t}^{(d+2)}(r) \, dr.
$$
  By continuity of $p_{t}^{(d+2)}(r)$ we arrive at (\ref{derivativeptr}).

\end{proof}

\section{The difference process}

Let $X$ be a pure-jump isotropic unimodal L{\'e}vy process in $\R^d$ with  an infinite  L{\'e}vy measure $\nu(dx) = \nu(x) \, dx$. The process has the transition density $p_t(x)$, which as a function of $x$ is also radially nonincreasing. 
We will use the following notation $\hat{x} = (-x_1,x_2,\ldots,x_d)$ for $x = (x_1,x_2,\ldots,x_d)$, $D_+ = \{(x_1,x_2,\ldots,x_d) \in D: \, x_1 > 0\}$, $D_- = \{(x_1,x_2,\ldots,x_d) \in D: \, x_1 < 0\}$ for $D \subset \R^d$. 

The aim of this section is to construct a Markov process $\tilde{X}_t$ on $\R^d_+$ with a sub-Markov transition density $p_{t}(x - y) - p_t(\hat{x} - y)$ and derive its basic properties. We call $\tilde{X}_t$ {\it{the difference process}}. First, we briefly present the construction of this process when $X_t$ is subordinate Brownian motion. In such case  the construction is easy and intuitive. Then we present the construction in the general case. 

First, let us assume that $X_t = B_{S_t}$, where $B_t$ is the Brownian motion in $\R^d$ (with the generator $\Delta$) and $S_t$ is a subordinator independent of $B_t$ with the Laplace exponent $\phi(\lambda)$. Denote by $g_t(x)$ the transition density of $B_t$. The transition density of $X_t$ is given by $p_t(x) = \int_0^{\infty} g_s(x) P(S_t \in ds)$. Let $\tau_{\R^d_+}^B = \inf\{t \ge 0: \, B_t \notin \R^d_+\}$  and $\tilde{B}_t$ be the Brownian motion killed on exiting $\R^d_+$ that is 
\begin{equation*}
\tilde{B}_t = 
\left\{
\begin{array}{ll}
\displaystyle
B_t, & \quad \text{for} \quad t < \tau_{\R^d_+}^B \\
\displaystyle
\partial, & \quad \text{for} \quad t \ge \tau_{\R^d_+}^B.
\end{array}
\right.
\end{equation*}  
Here we augment $\R_+^d$ by 
an extra point $\{\partial\}$ so that $\R^d_+ \cup \{\partial\}$ is a one-point compactification of $\R^d_+$. The sub-Markov transition density of $\tilde{B}_t$ on $\R^d_+$ is given by $g_t(x-y) - g_t(\hat{x}-y)$. Now let us put $\tilde{X}_t = \tilde{B}_{S_t}$. The sub-Markov transition density of $\tilde{X}_t$ on $\R^d_+$ is given by $\int_0^{\infty} (g_s(x-y) - g_s(\hat{x}-y)) P(S_t \in ds) = p_t(x-y) - p_t(\hat{x}-y)$.

Now let us consider the general case i.e. let $X$ be a pure-jump isotropic unimodal  L{\'e}vy process in $\R^d$ with  an infinite L{\'e}vy measure $\nu(dx) = \nu(x) \, dx$ and a transition density $p_t(x)$.

For any $t > 0$, $x,y \in R^d_+$ put
$$
\tilde{p}_t(x,y) = p_{t}(x - y) - p_t(\hat{x} - y).
$$

\begin{lemma}
\label{Chapman}
For any $s, t > 0$, $x,z \in \R^d_+$ we have
\begin{equation}
\label{Chapman1}
\int_{\R^d_+} \tilde{p}_t(x,y) \tilde{p}_s(y,z) \, dy = \tilde{p}_{t+s}(x,z).
\end{equation}
\end{lemma}
\begin{proof}
The left-hand side of (\ref{Chapman1}) equals
\begin{eqnarray*}
&& \int_{\R^d_+} {p}_t(x-y) {p}_s(y-z) \, dy - \int_{\R^d_+} {p}_t(x-y) {p}_s(\hat{y}-z) \, dy - 
\int_{\R^d_+} {p}_t(\hat{x}-y) {p}_s({y}-z) \, dy \\ && + \int_{\R^d_+} {p}_t(\hat{x}-y) {p}_s(\hat{y}-z) \, dy 
= \text{I} - \text{II} - \text{III} + \text{IV}.
\end{eqnarray*}
It is easy to check that 
\begin{eqnarray*}
\text{I} &=& p_{t+s}(x-z) - \int_{\R^d_-} {p}_t(x-y) {p}_s(y-z) \, dy,\\
\text{II} &=& p_{t+s}(\hat{x}-z) - \int_{\R^d_-} {p}_t(x-y) {p}_s(y-\hat{z}) \, dy,\\
\text{III} &=&  \int_{\R^d_-} {p}_t(x-y) {p}_s(y-\hat{z}) \, dy,\\
\text{IV} &=&  \int_{\R^d_-} {p}_t(x-y) {p}_s(y-z) \, dy,
\end{eqnarray*}
which implies the lemma.
\end{proof}

Let $C_0$ be the space of all continuous functions on $\R^d_+$ vanishing at $\partial \R^d_+$ and $\infty$ that is $f \in C_0$ iff $f \in \R^d_+ \to \R$ is continuous and for any $\varepsilon > 0$ there exist $\delta > 0$ and $M > 0$ such that for any $x \in \R^d_+$ if $\dist(x,\partial \R^d_+) < \delta$ or $|x| > M$ then $|f(x)| < \varepsilon$.

For any $f \in C_0$ and $t > 0$ put $\tilde{P}_t f(x) = \int_{\R^d_+} \tilde{p}_t(x,y) f(y) \, dy$, $\tilde{P}_0 f(x) = f(x)$. By Lemma \ref{Chapman} $\tilde{P}_t$ is a semigroup. Extend $f$ by putting $f(x) = -f(\hat{x})$ for $x \in \R_{-}^d$ and $f(x) = 0$ for $x \in \partial \R_+^d$. Note that $\tilde{P}_t f(x) = P_t f(x)$, $x \in R_+^d$ where $P_t f(x) = \int_{\R^d} p_t(x,y) f(y) \, dy$. Using this observation one can show that $\tilde{P}_t C_0 \subset C_0$ and the semigroup $\tilde{P}_t$ is strongly continuous in $t \ge 0$.

Now let us define $\tilde{P}_t(x,A)$, $t \ge 0$, $x \in \R^d_+$, $A \in \calB(\R^d_+)$ by $\tilde{P}_t(x,A) = \int_A \tilde{p}_t(x,y) \, dy$, $t > 0$, and $\tilde{P}_0(x,\cdot) = \delta_x$. By Lemma \ref{Chapman} $\tilde{P}_t(x,A)$ is a sub-Markov transition function on $\R^d_+$.

Let us augment $\R^d_+$ by an extra point $\partial$ so that $\R^d_+ \cup \{\partial\}$ is a one-point compactification of $\R^d_+$. We extend $\tilde{P}_t(x,A)$ to a Markov transition function on $\R^d_+ \cup \{\partial\}$ by setting 
\begin{equation}
\label{cemetary}
\tilde{P}_t(x,A) = 
\left\{
\begin{array}{ll}
\displaystyle
\tilde{P}_t(x,A \cap \R^d_+) + 1_A(\partial) (1 - \tilde{P}_t(x,\R^d_+)), & \quad \text{for} \quad x \in \R^d_+, \\
\displaystyle
1_A(\partial), & \quad \text{for} \quad x = \partial,
\end{array}
\right.
\end{equation} 
for any $A \subset \R^d_+ \cup \{\partial\}$ which is in the $\sigma$-algebra in $\R^d_+ \cup \{\partial\}$ generated by $\calB(\R^d_+)$. Then by standard results (see e.g. \cite[Chapter 1, Theorem 9.4]{BG1968}) there exists a Hunt process $\{\tilde{X}_t\}_{t \ge 0}$ with the state space $\R^d_+$ (augmented by $\{\partial\}$) and the transition function $\tilde{P}_t(x,A)$. We will denote by $\tilde{P}^x$, $\tilde{E}^x$ the probability and the expected value of the process $\tilde{X}_t$ starting from $x$.

Let us note $\tilde{p}_t(x,y) = \tilde{p}_t(y,x)$, $t > 0$, $x,y \in \R^d_+$. Put $\tau_D = \inf\{t > 0: \, \tilde{X}_t \notin D\}$.
\begin{lemma}
\label{regular}
Let $D \subset \R_+^d$ be an open, nonempty set and $z \in \partial D \cap \R^d_+$. If there exists a cone $A$ with vertex $z$ such that $A \cap B(z,r) \subset D^c$ for some $r > 0$ then $\tilde{P}^z(\tau_D = 0) = 1$.
\end{lemma}
\begin{proof}
Since $z \in \partial D \cap \R^d_+$ we may assume that $\dist(A \cap B(z,r),\R^d_-) > 0$. We have under $\tilde{P}^z$, $\{\tau_D = 0\} \supset \limsup_{n \to \infty} \{\tilde{X}_{1/n} \in A \cap B(z,r)\}$. Hence
\begin{equation*}
\tilde{P}^z(\tau_D = 0) \ge
\tilde{P}^z(\limsup_{n \to \infty} \{\tilde{X}_{1/n} \in A \cap B(z,r)\}) \ge
\limsup_{n \to \infty} \tilde{P}^z( \tilde{X}_{1/n} \in A \cap B(z,r)).
\end{equation*}
We have 
$$
\tilde{P}^z( \tilde{X}_{1/n} \in A \cap B(z,r)) =
{P}^z( {X}_{1/n} \in A \cap B(z,r)) - {P}^{\hat{z}}( {X}_{1/n} \in A \cap B(z,r)).
$$
By the rotational invariance and right-continuity of paths of $X$ there exists $\delta = \delta(A) > 0$ such that 
$
\limsup_{n \to \infty} {P}^z( {X}_{1/n} \in A \cap B(z,r)) \ge \delta
$.
Again by right-continuity of paths of $X$ and the fact that $\dist(A \cap B(z,r),\R^d_-) > 0$ we have
$
\limsup_{n \to \infty} {P}^{\hat{z}}( {X}_{1/n} \in A \cap B(z,r)) =0
$.
Hence
$$
\tilde{P}^z(\tau_D = 0) \ge \limsup_{n \to \infty} \tilde{P}^z( \tilde{X}_{1/n} \in A \cap B(z,r)) \ge \delta.
$$
Note that the Blumenthal's zero-or-one law holds for $\tilde{X}$. Hence $\tilde{P}^z(\tau_D = 0) = 1$.
\end{proof}

We say that $D \subset \R^d$ satisfies the {\it{outer cone condition}} if for any $z \in \partial D$ there exist $r > 0$ and a cone $A$ with vertex $z$ such that $A \cap B(z,r) \subset D^c$.

Let $D \subset \R^d_+$ be an open, nonempty set satisfying the outer cone condition. For any $t > 0$, $x,y \in D$ we put 
$$
\tilde{p}_D(t,x,y) = \tilde{p}_t(x,y) - \tilde{E}^x\left(\tilde{p}_{t - \tau_D}(\tilde{X}(\tau_D),y), t > \tau_D\right).
$$
It is easy to note that for any fixed $t > 0$, $x \in D$ the function $y \to \tilde{p}_D(t,x,y)$ is continuous in  $D \setminus \{x\}$. Using standard arguments (see e.g. \cite[Chapter II]{CZ1995}) one can show that for any Borel $A \subset D$, $x \in D$ and $t > 0$
\begin{equation}
\label{PtauD}
\tilde{P}^x(\tilde{X}_t \in A, \tau_D > t) = \int_A \tilde{p}_D(t,x,y) \, dy.
\end{equation}
Again using standard arguments and Lemma \ref{regular} we obtain
\begin{equation}
\label{tnD}
\tilde{P}^x(\tilde{X}_t \in A, \tau_D > t) = 
\lim_{n \to \infty} \tilde{P}^x\left(\tilde{X}_{\frac{t}{n}} \in D, \ldots, \tilde{X}_{\frac{(n-1)t}{n}} \in D, \tilde{X}_{t} \in A\right).
\end{equation}

We say that a set $D \subset \R^d$ is symmetric if for any $x \in D$ we have $\hat{x} \in D$.

\begin{lemma}
\label{Pdifference}
Assume that $D \subset \R^d$ is an open, symmetric, nonempty set satisfying the outer cone condition,  $x \in D_+$, $0 < t_1 < \ldots < t_n$, $n \in \N$, $A \subset D_+$. Then we have
\begin{eqnarray*}
&&\tilde{P}^x\left(\tilde{X}_{t_1} \in D_+, \ldots, \tilde{X}_{t_{n -1}} \in D_+, \tilde{X}_{t_n} \in A\right)\\
&=&  {P}^x\left({X}_{t_1} \in D, \ldots {X}_{t_{n -1}} \in D, {X}_{t_n} \in A\right)
- {P}^{\hat{x}}\left({X}_{t_1} \in D, \ldots, {X}_{t_{n -1}} \in D, {X}_{t_n} \in A\right).
\end{eqnarray*}
\end{lemma}
\begin{proof}
We will prove it by induction. For $n=1$ we have
$$
\tilde{P}^x\left(\tilde{X}_{t_1} \in A\right) = \int_A p_{t_1}(x-y) - p_{t_1}(\hat{x}-y) \, dy 
= P^x(X_{t_1} \in A) - P^{\hat{x}}(X_{t_1} \in A).
$$
Assume that the assertion of the lemma holds for $n$, we will show it for $n + 1$. Let $0 < t_1 < \ldots < t_n < t_{n+1}$. By the Markov property for $\tilde{X}_t$ we have
\begin{eqnarray}
\nonumber
&&\tilde{P}^x\left(\tilde{X}_{t_1} \in D_+, \ldots, \tilde{X}_{t_{n}} \in D_+, \tilde{X}_{t_{n+1}} \in A\right)\\
\label{Dtn1}
&=&  \tilde{E}^x\left(\tilde{X}_{t_1} \in D_+,  \, \, \tilde{P}^{\tilde{X}_{t_1}}\left(\tilde{X}_{t_2 - t_1} \in D_+, \ldots, \tilde{X}_{t_{n} - t_1} \in D_+, \tilde{X}_{t_{n+1} - t_1} \in A\right)\right).
\end{eqnarray}
For any $x \in \R^d$ put $f(x) = {P}^{x}({X}_{t_2 - t_1} \in D, \ldots, {X}_{t_{n} - t_1} \in D, {X}_{t_{n+1} - t_1} \in A)$.
By our induction hypothesis (\ref{Dtn1}) equals
\begin{eqnarray}
\label{Dtn1a}
&& \tilde{E}^x\left(\tilde{X}_{t_1} \in D_+, \, \,  f\left(\tilde{X}_{t_1}\right)\right) - \tilde{E}^x\left(\tilde{X}_{t_1} \in D_+, \, \, f\left(\widehat{\tilde{X}_{t_1}}\right)\right) \\
\nonumber
&=& \int_{D_+} (p_{t_1}(x-y) - p_{t_1}(\hat{x}-y)) (f(y) -f(\hat{y})) \, dy \\
\nonumber
&=& \int_{D_+} p_{t_1}(x-y) f(y) \, dy - \int_{D_+} p_{t_1}(x-y) f(\hat{y}) \, dy \\
\nonumber
&& - \int_{D_+} p_{t_1}(\hat{x}-y) f(y) \, dy + \int_{D_+} p_{t_1}(\hat{x}-y) f(\hat{y}) \, dy.
\end{eqnarray}
It is easy to verify that 
$$
\int_{D_+} p_{t_1}(x-y) f(\hat{y}) = \int_{D_-} p_{t_1}(\hat{x}-y) f(y) \, dy,
$$
$$
\int_{D_+} p_{t_1}(\hat{x}-y) f(\hat{y}) \, dy = \int_{D_-} p_{t_1}(x-y) f(y) \, dy.
$$
So (\ref{Dtn1a}) equals
\begin{eqnarray*}
&& \int_{D} p_{t_1}(x-y) f(y) \, dy - \int_{D} p_{t_1}(\hat{x}-y) f({y}) \, dy\\
&=& {E}^x\left({X}_{t_1} \in D,  \, \, f({X}_{t_1})\right) - {E}^{\hat{x}}\left({X}_{t_1} \in D,  \, \, f({X}_{t_1})\right)\\
&=& {P}^x\left({X}_{t_1} \in D, \ldots, {X}_{t_{n}} \in D, {X}_{t_{n+1}} \in A\right)
- {P}^{\hat{x}}\left({X}_{t_1} \in D, \ldots, {X}_{t_{n}} \in D, {X}_{t_{n+1}} \in A\right).
\end{eqnarray*}
\end{proof}

Let $D \subset \R^d$ be an open, nonempty, symmetric set satisfying the outer cone condition . Using the above lemma, (\ref{tnD}), (\ref{PtauD}) and continuity of $y \to \tilde{p}_{D_+}(t,x,y)$ on $D \setminus \{x\}$ we obtain that for any $t > 0$, $x,y \in D_+$, we have
$$
\tilde{p}_{D_+}(t,x,y) = p_D(t,x,y) - p_D(t,\hat{x},y).
$$
It follows that $\tilde{p}_{D_+}(t,x,y) \le p_D(t,x,y)$.

Now let $D \subset \R^d$ be an open, bounded, nonempty, symmetric set. For any $x \in D_+$ we have $\tilde{E}^x(\tau_{D_+}) < \infty$. Indeed, 
$$
\tilde{E}^x(\tau_{D_+}) = \int_0^{\infty} \int_{D_+} \tilde{p}_{D_+}(t,x,y) \, dy \, dt \le \int_0^{\infty} \int_{D} {p}_{D}(t,x,y) \, dy \, dt = {E}^x(\tau_{D}) < \infty.
$$
For $x,y \in D_+$ we define the Green function for $\tilde{X}_t$ and $D_+$ by
$\tilde{G}_{D_+}(x,y) = \int_0^{\infty} \tilde{p}_{D_+}(t,x,y) \, dt$. For any $x,y \in D_+$, $x \ne y$ we have
$$
0 < \tilde{G}_{D_+}(x,y) = G_D(x,y) - G_D(\hat{x},y) < G_D(x,y).
$$
Moreover, by $\tilde{p}_{D_+}(t,x,y)\le \tilde{p}(t,x,y)$, we have a trivial bound 

$$
0 < \tilde{G}_{D_+}(x,y) \le \int_0^{\infty} \tilde{p}(t,x,y) \, dt.
$$
Using standard arguments for any Borel, bounded $f:D_+ \to \R$ we have
$$
\tilde{E}^{x} \int_0^{\tau_{D_+}} f(\tilde{X}_t) \, dt = 
\int_D \tilde{G}_{D_+}(x,y) f(y) \, dy, \quad x \in D_+.
$$
For any $x,y \in D_+$, $x \ne y$ and a Borel set $A \subset \R^d_+$ put
$$
\tilde{\nu}(x,y) = \lim_{t \to 0} \frac{\tilde{p}_t(x,y)}{t} = \nu(x-y) - \nu(\hat{x}-y)
$$
and $\tilde{\nu}(x,A) = \int_A \tilde{\nu}(x,y) \, dy$. We call $\tilde{\nu}(x,A)$ the L{\'e}vy measure for the process $\tilde{X}$.

Let $D \subset \R^d$ be an open, bounded, nonempty, symmetric set, $x \in D_+$ and $A \subset \R^d_+ \setminus \overline{D}$ be a Borel set. Then by \cite[Theorem 1]{IW1962} we have
\begin{equation}
\label{tildeIW}
\tilde{P}^x\left(\tilde{X}(\tau_{D_+}) \in A \right) = 
\int_{D_+} \tilde{G}_{D_+}(x,y) \int_A \tilde{\nu}(y,z) \, dz \, dy.
\end{equation}
If additionally $\dist(D_+,\partial \R_+^d)> 0$ then again by (\ref{cemetary}) and \cite[Theorem 1]{IW1962} we have
\begin{equation}
\label{tildeIWcemetary}
\tilde{P}^x\left(\tilde{X}(\tau_{D_+}) \in \partial \right) = 
\int_{D_+} \tilde{G}_{D_+}(x,y)  \left(\int_{\R^d_-} \nu(y-z) \, dz + \int_{\R^d_+} \nu(\hat{y} - z) \, dz \right) \, dy.
\end{equation}

Now our aim is to show that for sufficiently regular $D$ we have $\tilde{P}^x\left(\tilde{X}(\tau_{D_+}) \in \R^d_+ \cap \partial D_+\right) = 0$ for $x \in D_+$.

We need to define an auxiliary family of stopping times:
\begin{eqnarray*}
T_{-1} &=& 0,\\
T_{2n} &=& \tau_{D_{+}} \circ \theta_{T_{2n-1}} + T_{2n-1}, \quad \quad n \ge 0,\\
T_{2n+1} &=& \tau_{D_{-}} \circ \theta_{T_{2n}} + T_{2n}, \quad \quad \,\,\, n \ge 0,
\end{eqnarray*}
Heuristically, up to time $\tau_D$ we count consecutive jumps from $D_+$ to $D_-$ and from $D_-$ to $D_+$. $T_{0}$ equals the first exit time of the process from $D_+$, if at $T_0$ the process jumps to $D_-$ then $T_1$ is the first exit time after $T_0$ from $D_-$. If at $T_{2n}$ the process jumps to $D_-$ then $T_{2n+1}$ is the first exit time after $T_{2n}$ from $D_-$. If at $T_{2n+1}$ the process jumps to $D_+$ then $T_{2n+2}$ is the first exit time of the process from $D_+$. If at some $T_{k}$ the process jumps to $D^c$ then all $T_m = \tau_D$ for $m \ge k$.

\begin{lemma}
\label{familyTn}
Let $D \subset \R^d$ be an open, bounded, nonempty, symmetric set, $x \in D_+$ and $A_+ \subset D_+$, $A_- \subset D_-$ be Borel sets. Assume that  $P^y(X(\tau_D) \in \partial D) = 0$ for any $y \in D$.  Then for any $n \ge 0$ we have
\begin{eqnarray}
\label{jumpplus}
P^x(X(T_{2n}) \in A_-) &=& \int_{D_+} p_{2n}(x,y) \int_{A_-} \nu(y-z) \, dz \, dy,\\
\label{jumpminus}
P^x(X(T_{2n+1}) \in A_+) &=& \int_{D_-} p_{2n+1}(x,y) \int_{A_+} \nu(y-z) \, dz \, dy,\\
\label{p2n}
E^x\left(\int_{T_{2n-1}}^{T_{2n}} 1_{A_+}(X_t) \, dt\right) &=& \int_{A_+} p_{2n}(x,w) \, dw,\\ 
\label{p2n1}
E^x\left(\int_{T_{2n}}^{T_{2n+1}} 1_{A_-}(X_t) \, dt\right) &=& \int_{A_-} p_{2n+1}(x,w) \, dw, 
\end{eqnarray}
where
\begin{eqnarray}
\nonumber
&& p_0(x,w) = G_{D_+}(x,w),\\
\label{p2n1formula}
&& p_{2n+1}(x,w) = \int_{D_+} p_{2n}(x,y) \int_{D_-} \nu(y-z) G_{D_-}(z,w) \, dz \, dy, \quad w \in D_-, \, n \ge 0, \quad \quad\\
\label{p2nformula}
&& p_{2n}(x,w) = \int_{D_-} p_{2n-1}(x,y) \int_{D_+} \nu(y-z) G_{D_+}(z,w) \, dz \, dy, \quad  w \in D_+, \, n \ge 1. \quad \quad
\end{eqnarray}
\end{lemma}
\begin{proof}
We prove the lemma by induction. The case $n=0$ is left to the Reader. Assume that (\ref{jumpplus}), (\ref{jumpminus}), (\ref{p2n}), (\ref{p2n1}) hold for some $n \ge 0$. We will show it for $n+1$. By the strong Markov property we obtain
\begin{equation}
\label{jumpplus2}
P^x(X(T_{2n+2}) \in A_-) =
E^x \left(P^{X(T_{2n+1})}\left(X(\tau_{D_+}) \in A_-\right), \, X(T_{2n+1}) \in D_+\right).
\end{equation}
Now the Ikeda-Watanabe formula (\ref{IW}) and the induction hypothesis (\ref{jumpminus})  give that (\ref{jumpplus2}) equals
\begin{eqnarray*} 
&& E^x \left( \int_{D_+} G_{D_+}(X(T_{2n+1}),w) \int_{A_-} \nu(w-v) \, dv \, dw , \, X(T_{2n+1}) \in D_+ \right)\\
&=& \int_{D_-} p_{2n+1}(x,y) \int_{D_+} \nu(y-z)
\int_{D_+} G_{D_+}(z,w) \int_{A_-} \nu(w-v) \, dv \, dw \, dz \, dy\\
&=&
\int_{D_+} p_{2n+2}(x,w) \int_{A_-} \nu(w-v) \, dv \, dw,
\end{eqnarray*}
which gives (\ref{jumpplus}) for $n+1$. Again by the strong Markov property we get
\begin{eqnarray*}
E^x\left(\int_{T_{2n+1}}^{T_{2n+2}} 1_{A_+}(X_t) \, dt\right)
&=& E^x\left(\left(\int_{0}^{\tau_{D_+}} 1_{A_+}(X_t) \, dt\right)\circ \theta_{T_{2n+1}}, \, X(T_{2n+1}) \in D_+\right)\\
&=& E^x\left(E^{X(T_{2n+1})}\left(\int_{0}^{\tau_{D_+}} 1_{A_+}(X_t) \, dt\right), \, X(T_{2n+1}) \in D_+\right)\\
&=& E^x\left(\int_{A_+} G_{D_+}(X(T_{2n+1}),w) \, dw, \, X(T_{2n+1}) \in D_+\right).
\end{eqnarray*}
By the induction hypothesis (\ref{jumpminus}) this is equal to
$$
\int_{D_-} p_{2n+1}(x,y) \int_{D_+} \nu(y-z) \int_{A_+} G_{D_+}(z,w) \, dw \, dz \, dy = \int_{A_+} p_{2n+2}(x,w) \, dw.
$$
This shows (\ref{p2n}) for $n+1$. The proof of (\ref{jumpminus}) and (\ref{p2n1}) for $n+1$ is analogous and it is omitted.
\end{proof}

\begin{lemma}
\label{boundarytilde}
Let $D \subset \R^d$ be an open, bounded, nonempty, symmetric set such that  $P^y(X(\tau_D) \in \partial D) = 0$ for any $y \in D$. Then we have
$$
\tilde{P}^x\left(\tilde{X}(\tau_{D_+}) \in \partial D_+\cap\R^d_+\right) = 0, \quad \quad x \in D_+.
$$
\end{lemma}
\begin{proof} First, we prove the lemma under the assumption  $\dist(D_+,\partial \R^d_+) > 0$.
Note that   $P^y(X(\tau_D) \in \partial D) = 0$ for any $y \in D$ yields   $P^y(X(\tau_{D_+}) \in \partial {D_+}) = 0$ for any $y \in D_+$. Moreover,
 our assumptions imply that the Lebesgue measure of both $\partial{D}$ and $\partial{D_+}$ is zero.
By (\ref{tildeIW}) we have
$$
\text{I} = \tilde{P}^x\left(\tilde{X}(\tau_{D_+}) \in (D^c)_+ \setminus \partial D_+\right) = \int_{D_+} \tilde{G}_{D_+}(x,y) \int_{(D^c)_+} \tilde{\nu}(y,z) \, dz \, dy.
$$
By (\ref{tildeIWcemetary}) we have
$$
\text{II} = \tilde{P}^x\left(\tilde{X}(\tau_{D_+}) \in \{\partial\}\right) =
\int_{D_+} \tilde{G}_{D_+}(x,y) \left(\int_{\R^d_-} \nu(y-z) \, dz + \int_{\R^d_+} \nu(\hat{y}-z) \, dz\right)\, dy.
$$
Hence
\begin{eqnarray*}
\text{I} + \text{II} &=& \tilde{P}^x\left(\tilde{X}(\tau_{D_+}) \in \{\partial\} \cup (D^c)_+ \setminus \partial D_+\right)\\
&=& \int_{D_+} \tilde{G}_{D_+}(x,y) \int_{D^c} \nu(y-z) \, dz \, dy \\
&+& \int_{D_+} \tilde{G}_{D_+}(x,y) \left(\int_{D_-} \nu(y-z) \, dz + \int_{D_+} \nu(\hat{y}-z) \, dz \right) \, dy\\
&=& \text{III} + \text{IV}.
\end{eqnarray*}

Note that $\text{I} + \text{II} = 1 - \tilde{P}^x\left(\tilde{X}(\tau_{D_+}) \in \partial D_+\right)$. So it is enough to show that 
\begin{equation}
\label{III+IV}
\text{III} + \text{IV} = 1.
\end{equation}
Let $X(\tau_D)_* = \lim_{t \nearrow \tau_{D}} X(t)$.
We have
\begin{eqnarray}
\nonumber
\text{III} &=&
\int_{D_+} G_{D}(x,y) \int_{D^c} \nu(y-z) \, dz \, dy
- \int_{D_+} G_{D}(\hat{x},y) \int_{D^c} \nu(y-z) \, dz \, dy\\
\nonumber
&=& \int_{D} G_{D}(x,y) \int_{D^c} \nu(y-z) \, dz \, dy 
- \int_{D_-} G_{D}(x,y) \int_{D^c} \nu(y-z) \, dz \, dy \\
&& \quad \quad \quad \quad \quad - \int_{D_-} G_{D}(\hat{x},\hat{y}) \int_{D^c} \nu(\hat{y}-z) \, dz \, dy\\
\nonumber
&=& 1 - 2 \int_{D_-} G_{D}(x,y) \int_{D^c} \nu(y-z) \, dz \, dy\\
\label{IIIformulaA}
&=& 1 - 2 P^x(X(\tau_D)_* \in D_-).
\end{eqnarray}
We also have
\begin{eqnarray*}
\text{IV} &=&
2 \int_{D_+} (G_{D}(x,y)-G_D(\hat{x},y))\int_{D_-} \nu(y-z) \, dz \, dy\\
&=& 2 \left( \int_{D_+} G_{D}(x,y) \int_{D_-} \nu(y-z) \, dz \, dy 
- \int_{D_-} G_{D}(x,y) \int_{D_+} \nu(y-z) \, dz \, dy \right).
\end{eqnarray*}
Note that by Lemma \ref{familyTn} we get
\begin{eqnarray*}
G_D(x,y) &=& \sum_{n = 0}^{\infty} p_{2n}(x,y), \quad \quad \text{for} \, \, y \in D_+,\\
G_D(x,y) &=& \sum_{n = 0}^{\infty} p_{2n+1}(x,y), \quad \quad \text{for} \, \, y \in D_-.
\end{eqnarray*}
It follows that 
\begin{equation}
\label{IVexpansion}
\text{IV} = 2 \left( \sum_{n = 0}^{\infty} \int_{D_+} p_{2n}(x,y) \int_{D_-} \nu(y-z) \, dz \, dy 
- \sum_{n = 0}^{\infty} \int_{D_-} p_{2n+1}(x,y) \int_{D_+} \nu(y-z) \, dz \, dy \right).
\end{equation}

Note that for any $z \in D_-$ we have
$$
P^z(X(\tau_{D_-}) \in (D_-)^c) = \int_{D_-} G_{D_-}(z,w) \int_{(D_-)^c} \nu(w-q) \, dq \, dw = 1.
$$
Using this and (\ref{p2n1formula}) we get
\begin{eqnarray*}
&& \int_{D_+} p_{2n}(x,y) \int_{D_-} \nu(y-z) \, dz \, dy\\
&=& \int_{D_+} p_{2n}(x,y) \int_{D_-} \nu(y-z) \int_{D_-} G_{D_-}(z,w) \int_{(D_-)^c} \nu(w-q) \, dq \, dw \, dz \, dy\\
&=& \int_{D_-}\left[\int_{D_+} p_{2n}(x,y) \int_{D_-} G_{D_-}(z,w) \nu(y-z) \, dz \, dy\right]   \int_{(D_-)^c} \nu(w-q) \, dq \, dw \\
&=& \int_{D_-} p_{2n+1}(x,w) \int_{(D_-)^c} \nu(w-q) \, dq \, dw\\
&=& \int_{D_-} p_{2n+1}(x,w) \int_{D^c} \nu(w-q) \, dq \, dw 
+ \int_{D_-} p_{2n+1}(x,w) \int_{D_+} \nu(w-q) \, dq \, dw.
\end{eqnarray*}
Substituting this to (\ref{IVexpansion}) we get
\begin{eqnarray*}
\text{IV} &=& 2 \left( \sum_{n = 0}^{\infty} \int_{D_-} p_{2n+1}(x,w) \int_{D^c} \nu(w-q) \, dq \, dw \right)\\
&=& 2  \int_{D_-} G_D(x,w) \int_{D^c} \nu(w-q) \, dq \, dw\\
&=& 2 P^x(X(\tau_D)_* \in D_-).
\end{eqnarray*} 
Combining the last equality with (\ref{IIIformulaA}) we obtain (\ref{III+IV}), which completes the proof in the case  $\dist(D_+,\partial \R^d_+) > 0$ .

To remove the above condition, 
for any $\ve > 0$, we consider  $D_{\ve} = \{(y_1,\ldots,y_d) \in D_+: \, y_1 > \ve\}$ and $H_{\ve} = \{(y_1,\ldots,y_d): \, y_1 > \ve\}$. From (\ref{Xboundary}) we infer  that $P^y(X(\tau_{H_{\ve}}) \in \partial H_{\ve}) = 0$ for any $\ve > 0$ and $y \in H_{\ve}$. This implies that for any $\ve > 0$ and $y \in D_{\ve}$ we have $P^y(X(\tau_{D_{\ve}}) \in \partial D_{\ve}) = 0$.

Fix $x \in D_+$. There exists $\ve_1 > 0$ such that $x \in D_{\ve_{1}}$. For any $\ve \in (0,\ve_1]$ we have
$$
\tilde{P}^x\left(\tilde{X}(\tau_{D_+}) \in \partial D_+ \cap \R^d_+ \right) \le
\tilde{P}^x\left(\tilde{X}(\tau_{D_{\ve}}) 
\in \partial D_{\ve} \cup \left(D_+ \setminus \overline{D_{\ve}}\right) \right).
$$ 
By the first part of the proof the last probability is equal to $\tilde{P}^x\left(\tilde{X}(\tau_{D_{\ve}}) 
\in  D_+ \setminus \overline{D_{\ve}} \right)$. 
By (\ref{tildeIW}) we have
\begin{eqnarray*}
\tilde{P}^x\left(\tilde{X}(\tau_{D_{\ve}}) 
\in  D_+ \setminus \overline{D_{\ve}} \right)
&=& \int_{D_{\ve}} \tilde{G}_{D_{\ve}}(x,y) \int_{D_+ \setminus D_{\ve}} \tilde{\nu}(y,z) \, dz \, dy\\
&\le& \int_{D_{\ve}} G_{D_{\ve}}(x,y) \int_{D_+ \setminus D_{\ve}} \nu(y-z) \, dz \, dy\\
&\le& P^x(X(\tau_{H_{\ve}}) \in \R_+^d \setminus H_{\ve}).
\end{eqnarray*}
Clearly  this tends to $0$ as $\ve\searrow0$. Hence $\tilde{P}^x\left(\tilde{X}(\tau_{D_+}) \in \partial D_+ \cap \R^d_+ \right) = 0$.
\end{proof}

As a conlusion of (\ref{Xboundary}) and Lemma \ref{boundarytilde} we obtain 
\begin{corollary}
\label{boundarytilde3}
Let $D \subset \R^d$ be a symmetric, open, nonempty, bounded Lipschitz 
set. Then we have
$$
\tilde{P}^x\left(\tilde{X}(\tau_{D_+}) \in \partial D_+ \cap \R^d_+ \right) = 0, \quad \quad x \in D_+.
$$
\end{corollary}
It follows that under assumptions of the above corollary for a Borel set $A \subset \R^d_+ \setminus {D}$ and $x \in D_+$ we have 

\begin{equation}
\label{domin}
\tilde{P}^x\left(\tilde{X}(\tau_{D_+}) \in A \right) = 
\int_{D_+} \tilde{G}_{D_+}(x,y) \int_A \tilde{\nu}(y,z) \, dz \, dy
\le {P}^x\left({X}(\tau_{D}) \in A \right). 
\end{equation}

\section{Auxiliary estimates of the L{\'e}vy measure and the Green function}

Throughout this section we will assume that the process $X$ satisfies the assumptions (A). In fact it is enough to assume only  (H0) and (H1).

\begin{lemma}
\label{Levyquotient}
For any $r > 0$ we have
$$
\left|\frac{\nu'(r)}{\nu(r)}\right| \le  \left(3(a_1-1)\right)\frac1 {r\wedge1}.
$$

Moreover, for $0<r_1<r_2<\infty$,

$$
\frac{\nu(r_1)}{\nu(r_2)} \le  \left(\frac {r_2}{r_1}\right)^{3(a_1-1)}e^{3(a_1-1) (r_2-r_1)}
$$
and
$$\nu(r_1)-\nu(r_2)\le \frac32(a_1-1)\left(\frac {\nu(r_1)}{1\wedge r_1}\right)(r_2-r_1)\left(1+\frac {r_2}{r_1}\right).$$
\end{lemma}
\begin{proof} Let $0<u<v<\infty$. Then by absolute continuity of $\nu(\rho)$ and monotonicity of $-\nu^\prime(\rho)/\rho$ we have 
$$\nu(u)-\nu(v)= \int_u^v \rho\frac{-\nu^\prime (\rho)}\rho d\rho\ge \frac{-\nu^\prime (v)}v\int_u^v \rho d\rho = \frac{-\nu^\prime (v)}{2v}(v^2-u^2).  $$
Hence 
$$|\nu'(v)|\le 2v\frac{\nu(u) - \nu(v)} {v^2-u^2}.$$
Next, we take $v=r, u=r/2$ if $r\le 2$,  to arrive at 
$$\left|{\nu'(r)}\right|\le \frac 8{3r} (\nu(r/2) - \nu(r))\le \frac {8(a_1-1)}{3r} \nu(r).$$
Similarly for   $v=r+1, u=r$,  $r\ge 1 $,

$$\left|\frac{\nu'(r+1)}{\nu(r+1)}\right|\le \frac{4(a_1-1)}3.$$
Combining both estimates we complete the proof of the first assertion. The second one is an easy consequence of the first.

Again let $0<u<v<\infty$, then using monotonicity of $-\nu^\prime(\rho)/\rho$ and  the first claim of the lemma  we obtain 
\begin{eqnarray*}\nu(u)-\nu(v)&=& \int_u^v \rho\frac{-\nu^\prime (\rho)}\rho d\rho\le \frac{-\nu^\prime (u)}u\int_u^v \rho d\rho\\& =& \frac{-\nu^\prime (u)}{2u}(v^2-u^2)\le  \left(3(a_1-1)\right)\frac{\nu (u)} {u\wedge1} \frac{(v^2-u^2)}{2u}. \end{eqnarray*}
\end{proof}

By Lemma \ref{Levyquotient} and (\ref{Levymeasureestimates}) we obtain
\begin{corollary}
\label{ABLevyquotient}
 For any $v,z \in \R^d_+$, 

\begin{eqnarray*}
\tilde{\nu}(v,z) &\le&\frac32(a_1-1) |z-\hat{z}| \frac{\nu(v-z)}{1\wedge|v-z|} \left(1+\frac{|v-\hat{z}|}{|v-z|}\right)\\ &\le& c |z-\hat{z}| \frac{|v-\hat{z}|}{|v-z|^{d+1}(1\wedge|v-z|)\V^2(|v-z|)},
\end{eqnarray*}
where $c= (a_1-1)c_1(d)$.

\end{corollary}

By Lemma \ref{Levyquotient} we easily obtain:
\begin{corollary}
\label{boundedness}
If a measurable function $f:\R^d \to [0,\infty)$ satisfies $\int_{\R^d} f(y) (\nu(y) \wedge 1) \, dy < \infty$ then for any $x \in \R^d$ we have $\int_{\R^d} f(y) (\nu(x-y) \wedge 1) \, dy < \infty$.
\end{corollary}

\begin{lemma}
\label{Lipschitzg}
Let $w \in \R^d$, $r \in (0,2]$, put $B = B(w,r)$. Assume that a measurable function $f:\R^d \to [0,\infty)$ satisfies $\int_{\R^d} f(y) (\nu(y) \wedge 1) \, dy < \infty$. For $y \in B$ put $g(y) = \int_{B^c} f(z) \nu(y - z) \, dz$. Then the function $g$ is bounded on $B(w,r/2)$ and satisfies
$$
|g(y) - g(w)| \le c \frac{(g(w) \wedge g(y)) |y - w|}{r}, \quad \quad y \in B(w,r/2).
$$
We also have $g(w) \le c g(y)$ for $y \in B(w,r/2)$.
\end{lemma}
\begin{proof}
Let $y \in B(w,r/2)$.
By Corollary \ref{boundedness}, $g(y) < \infty$. If   $z \in B^c$, by Lemma \ref{Levyquotient} we have,

\begin{eqnarray*}|\nu(y-z) -\nu(w-z)|&\le& \nu(|z-w|-|y-w|) -\nu(|z-w|)\\
&\le& 12(a_1-1)\frac {\nu\left(|z-w|-|y-w|\right)} {r} |y-w|\\
&\le& \frac {\nu(|z-w|)} {r} |y-w| 12 (a_1-1) 2^{3(a_1-1)}e^{3(a_1-1)r/2}.
\end{eqnarray*}
Similarly, 
\begin{eqnarray*}|\nu(y-z) -\nu(w-z)|
&\le& \frac {\nu(|z-y|)} r |y-w| 9(a_1-1) 3^{3(a_1-1)}e^{3(a_1-1)r}.
\end{eqnarray*}
Combining both estimates we obtain
\begin{eqnarray*}|\nu(y-z) -\nu(w-z)|
&\le& \frac c {r}\left(\nu(|z-y|)\wedge \nu(|w-z|)\right) |y-w|,
\end{eqnarray*}
where $c=12(a_1-1) 3^{3(a_1-1)}e^{3(a_1-1)r}$.

It follows that 
\begin{eqnarray*}
|g(y) - g(w)| &\le& \int_{B^c} f(z) |\nu(y-z) -\nu(w-z)| \, dz\\
&\le& c \frac{ |y - w|}{r} \int_{B^c} f(z) (\nu(w-z) \wedge \nu(y-z)) \, dz\\
&\le& c
\frac{ (g(w) \wedge g(y)) |y - w|}{r}.
\end{eqnarray*}
\end{proof}

In what follows by $\{e_1, \dots, e_d\} $ we denote the standard orthonormal basis in $\R^d$.

\begin{proposition}
\label{Greenhatest}
 Let $r > 0$, $x_1 \in (0,r)$, put $B = B(0,r)$, $x = x_1 e_1$. Let $y \in B_+$ such that $|y| \ge 4 |x|$. Then we have
$$
0 < \tilde{G}_{B_+}(x,y) \le c |x - \hat{x}| | y| U^{(d+2)}\left(\frac{|y|}{2}\right) \le c |x - \hat{x}| \frac{\V^2(|y|)}{|y|^{d+1}},
$$ 
where $c = c(d)$.
\end{proposition}
\begin{proof}

By the Lagrange theorem there is a point $\xi$  between $\hat{x}$ and $x$ ($\xi$ depends on $t$, $x$, $y$ and the process $X$) such that

$$p_t(x-y) - p_t(\hat{x} - y)=|x - \hat{x}| \frac{\partial}{\partial \xi_1}p_t(\xi-y).$$
 By Theorem \ref{dertransition} this equals
\begin{equation}
\label{intder}
-2\pi|x - \hat{x}|  (\xi_1 - y_1) p_t^{(d+2)}(\xi-y).
\end{equation}
Note also that $\frac{1}{2} |y|\le |\xi - y|  \le \frac{3}{2} |y|$. 
 By Theorem \ref{dertransition}, $p_t^{(d+2)}$ is radial and radially nonincreasing. Hence by (\ref{intder}) we have
 $$|p_t(x-y) - p_t(\hat{x} - y)|\le  3 \pi |x - \hat{x}| | y|  p_t^{(d+2)}\left(\frac{|y|}{2}\right).$$

Next,\begin{eqnarray*}
\tilde{G}_{B+}(x,y) &\le&
\int_0^{\infty} |p_t(x-y) - p_t(\hat{x} - y)| \, dt \\
&\le& 3 \pi |x - \hat{x}| |y| \int_0^{\infty} p_t^{(d+2)}\left(\frac{|y|}{2}\right) \, dt \\&=&  3\pi |x - \hat{x}| |y| U^{(d+2)}\left(\frac{|y|}{2}\right).
\end{eqnarray*}
By \cite[Theorem 3]{G2013} the last expression  is bounded from above by 
$$
\frac{c |x - \hat{x}| \V^2\left(\frac{|x - y|}{2}\right)}{|x - y|^{d+1}} \le 
\frac{c |x - \hat{x}| \V^2(|x-y|)}{|x - y|^{d+1}},
$$
where $c = c(d)$.
\end{proof}

\begin{lemma}
\label{Greenintegral}
For any $r >0$, $h \in (0,r/16)$, $x = h e_1$, $B = B(0,r)$ we have
$$
\int_{B_+} \tilde{G}_{B_+}(x,y) |y| \, dy \le c |x| \int_{B(0,r/4)} G_B(x,y) \, dy.
$$
\end{lemma}
\begin{proof} It is obvious that
$$
\int_{ B(0,4h)_+} \tilde{G}_{B_+}(x,y) |y| \, dy \le 4 |x| \int_{B(0,r/4)} G_B(x,y) \, dy.
$$
Hence it is enough to estimate the integral over $(B \setminus B(0,4h))_+$.
 For any $y \in (B \setminus B(0,4h))_+$ we have $|y| \ge 2 |x|$. By Proposition \ref{Greenhatest} we get
\begin{eqnarray*}
\int_{(B \setminus B(0,4h))_+} \tilde{G}_{B_+}(x,y) |y| \, dy &\le& c |x - \hat{x}| \int_{(B \setminus B(0,4h))_+}   |y|^2 U^{(d+2)}(|y|/2)dy\\
&\le& c |x - \hat{x}| \int_{0}^r   \rho^{d+1} U^{(d+2)}( \rho)d\rho\\
&\le& c |x - \hat{x}|\V^2(r), 
\end{eqnarray*}
where the last inequality follows from 
\cite[Proposition  2]{G2013}.
 Finally, by Lemma \ref{Exit_ball} we obtain the conclusion. 
\end{proof}

\begin{lemma}
\label{ring}
For any   $r>0$  and any $y \in B(0,r) \setminus \overline{B(0,3r/4)}$ we have
$$
P^y(X_{\tau_R} \in B(0,r) \setminus R) \le c \frac{\V(\delta(y))}{\V(r)},
$$
where $R = B(0,r) \setminus \overline{B(0,r/2)}$, $\delta(y) = \delta_{B(0,r)}(y)$ and $c=c(d)$.
\end{lemma}
\begin{proof}
We may assume that $y = q e_1$ for some $q \in (3r/4,r)$. Put $z = r e_1$ and $D = B(z,r/2) \cap R$. Clearly, $y \in B(z,r/4)$ and 
\begin{equation*}
P^y(X(\tau_R) \in B(0,r) \setminus R) \le
P^y(X(\tau_D) \in B(0,r) \setminus D) 
\le P^y(X(\tau_D) \in B^c(z,r/2)).
\end{equation*} 
By Lemma \ref{UBHP} and then   by    Lemma \ref{Exit_ball} we obtain
$$
P^y(X(\tau_D) \in B^c(z,r/2)) \le c \frac{E^y(\tau_D)}{\V^2(r/2)} \le c \frac{E^y(\tau_{B(0,r)})}{\V^2(r)}\le c\frac{\V(\delta(y))}{\V(r)} ,
$$
where $c = c(d)$. 
\end{proof}

\begin{lemma}
\label{tildeGreen} For any $0<r\le 1$, $h \in (0,r/16)$, $x = h e_1$, $B = B(0,r)$ and $y \in B_+ \setminus B(0,r/4)_+$ we have
$$
 \tilde{G}_{B_+}(x,y) \le \frac{c h \V(\delta(y)) \V(r)}{r^{d+1}},
$$
where $\delta(y) = \delta_B(y)$.
\end{lemma}
\begin{proof}
Let us denote $R = B \setminus \overline{B(0,r/2)}$. By Proposition \ref{Greenhatest} we get
$$
\tilde{G}_{B_+}(x,y) \le \frac{c h \V(\delta(y)) \V(r)}{r^{d+1}},\quad y\in \overline{B(0,3r/4)} \setminus \overline{B(0,r/4)}.
$$
Hence may assume that $y \in B_+ \setminus B(0,3r/4)_+$.  Since $\tilde{G}_{B_+}(x,\cdot)$ is harmonic in $B_+\setminus\{x\}$ with respect to $\tilde{X}$  we have
\begin{eqnarray*}
\tilde{G}_{B_+}(x,y) &=&
\tilde{E}^y\left(\tilde{G}_{B_+}\left(x,\tilde{X}(\tau_{R_+})\right)\right) \\
&=&
\tilde{E}^y\left(\tilde{G}_{B_+}\left(x,\tilde{X}(\tau_{R_+})\right), \tilde{X}(\tau_{R_+}) \in B_+ \setminus (R_+ \cup B(0,r/4)_+)\right) \\
&& + \tilde{E}^y\left(\tilde{G}_{B_+}\left(x,\tilde{X}(\tau_{R_+})\right), \tilde{X}(\tau_{R_+}) \in B(0,r/4)_+\right) = \text{I} + \text{II}.
\end{eqnarray*}
 Since $ B_+ \setminus (R_+ \cup B(0,r/4)_+$ satisfies the assumptions of Corollary \ref{boundarytilde3} we can  apply (\ref{domin}) to obtain 
\begin{eqnarray*}
\text{I} &\le& \sup_{z \in B_+ \setminus (R_+ \cup B(0,r/4)_+)} \tilde{G}_{B_+}(x,z) \tilde{P}^y\left( \tilde{X}(\tau_{R_+}) \in B_+ \setminus (R_+ \cup B(0,r/4)_+)\right) \\
&\le& \sup_{z \in B_+ \setminus (R_+ \cup B(0,r/4)_+)} \tilde{G}_{B_+}(x,z) P^y\left( {X}_{\tau_{R}} \in B \setminus R\right).
\end{eqnarray*}
By Proposition  \ref{Greenhatest} and Lemma \ref{ring} this is bounded from above by
$
c h \V(\delta(y)) \V(r) r^{-d-1}
$.
By the Ikeda-Watanabe formula for $\tilde{X}$ (\ref{tildeIW})  we get
\begin{equation*}
\label{IIring}
\text{II} = \int_{R_+} \tilde{G}_{R_+}(y,v) \int_{B(0,r/4)_+} \tilde{\nu}(v,z) \tilde{G}_{B_+}(x,z) \, dz \, dv. 
\end{equation*}
Furthermore, by Corollary \ref{ABLevyquotient}  we have for $v \in R_+$, $z \in B(0,r/4)_+$,
$$
\tilde{\nu}(v,z) \le c |z-\hat{z}| \frac{1}{|v-z|^d(1\wedge|v-z|)\V^2(|v-z|)} \le \frac{c |z|}{r^{d+1} \V^2(r)}.
$$
 This combined with Lemma \ref{Greenintegral} and the   estimates of $E^x(\tau_B)$ from Lemma \ref{Exit_ball}  yields  
\begin{eqnarray*}
\text{II}&\le& \frac{c}{r^{d+1} \V^2(r)} \int_R G_R(y,v) \, dv \int_{B_+} |z| \tilde{G}_{B_+}(x,z) \, dz \\
&\le& \frac{c h}{r^{d+1} \V^2(r)} E^y(\tau_B) E^x(\tau_B)
\le \frac{c h \V(\delta(y)) \V(r)}{r^{d+1}}.
\end{eqnarray*}
\end{proof}

\section{Proof of the main theorem}
Throughout this section we will assume that the process $X$ satisfies the assumptions (A).
The following proposition is the key step in proving gradient estimates of harmonic functions for L{\'e}vy processes.
\begin{proposition}
\label{Lipschitz}
 Let $0<r <1/4$, $h \in (0,r/16)$, $x = h e_1$. Assume that $f:\R^d \to [0,\infty)$ is harmonic in $B(0,4r)$ with respect to $X$. Then we have
$$
f(x) - f(\hat{x}) \le c \frac{h f(0)}{r}.
$$
\end{proposition}
\begin{proof}
Put $B = B(0,r)$. For $y \in B$ put $g(y) = \int_{B^c} f(z) \nu(y-z) \, dz$. By harmonicity of $f$ and the Ikeda-Watanabe formula (\ref{IW}) we have $f(x) = \int_B G_B(x,y) g(y) \, dy$. Observe $g(y) < \infty$ a.e. on $B$. We have 
\begin{eqnarray*}
f(x) - f(\hat{x}) 
&=& \int_{B_+} (G_B(x,y) - G_B(\hat{x},y)) g(y) \, dy 
+ \int_{B_-} (G_B(x,y) - G_B(\hat{x},y)) g(y) \, dy \\
&=& \int_{B_+} (G_B(x,y) - G_B(\hat{x},y)) g(y) \, dy
+ \int_{B_+} (G_B(x,\hat{y}) - G_B(\hat{x},\hat{y})) g(\hat{y}) \, dy \\
&=& \int_{B_+} \tilde{G}_{B_+}(x,y) (g(y) - g(\hat{y})) \, dy.
\end{eqnarray*}
Hence $f(x) - f(\hat{x})$ is equal to
\begin{eqnarray*} 
&& 
 \int_{B(0,r/4)_+ } \tilde{G}_{B_+}(x,y) (g(y) - g(\hat{y})) \, dy \\
&& + \int_{B_+ \setminus B(0,r/4)_+} \tilde{G}_{B_+}(x,y) \int_{B^c(0,2r)} f(z) (\nu(y-z) - \nu(\hat{y}-z))\, dz \, dy \\
&& + \int_{B_+ \setminus B(0,r/4)_+} \tilde{G}_{B_+}(x,y) \int_{B(0,2r)\setminus B} f(z) (\nu(y-z) - \nu(\hat{y}-z)) \, dz \, dy \\
&=& \text{I} + \text{II} + \text{III}. 
\end{eqnarray*}

{By Lemma \ref{Lipschitzg} for $y \in B(0,r/4)_+$ we obtain $$|g(y) - g(\hat{y})| \le |g(y) - g(0)|+| g(\hat{y})-g(0)|\le c|y| r^{-1} g(0).$$} Lemma \ref{Greenintegral} and the above inequality yield
$$
\text{I} \le \frac{c g(0)}{r} \int_{B(0,r/4)_+} \tilde{G}_{B_+}(x,y)|y|\, dy \le \frac{c|x| g(0)}{r}\int_{B(0,r/4)} {G}_{B}(x,y)\, dy.
$$
Moreover, using again Lemma \ref{Lipschitzg}, we have $g(0)\le cg(y),\, y\in B(0,r/4)$, hence

$$
\text{I}\le \frac{c|x| }{r}\int_{B(0,r/4)} {G}_{B}(x,y)g(y)\, dy\le \frac{c|x|f(x) }{r}.
$$

$\text{II}$ will be estimated similarly like $\text{I}$. For $y \in B$ put $g_1(y) = \int_{B(0,2r)^c} f(z) \nu(y-z) \, dz\le g(y)$. By Lemma \ref{Lipschitzg} applied to $g_1$ we obtain 
\begin{equation*}
\label{IIformula}
\text{II} \le \frac{c g(0)}{r}\int_{B_+ \setminus B(0,r/4)_+ } \tilde{G}_{B_+}(x,y) |y| \, dy.
\end{equation*}
Repeating the same steps as used to estimate $\text{I}$ we obtain
\begin{equation*}
\text{II} \le \frac{c|x|f(x) }{r}.
\end{equation*}


Finally we estimate  $\text{III}$.  By the assumed  Harnack inequality  we obtain
\begin{eqnarray}
\nonumber
\text{III} &\le& c f(x) \int_{B_+ \setminus B(0,r/4)_+} \tilde{G}_{B_+}(x,y) \int_{B(0,2r)\setminus B}  \nu(y-z) \, dz \, dy \\
\label{IV1}
&\le& c f(x) \int_{B_+ \setminus B(0,r/4)_+} \tilde{G}_{B_+}(x,y) \int_{B^c(y,\delta(y))}  \nu(y-z) \, dz \, dy,
\end{eqnarray}
where $\delta(y) = \delta_B(y)$. Denote $R(\rho)= \int_{B^c(0,\rho)} \nu(x)dx$. 

By Lemma \ref{tildeGreen} we obtain that (\ref{IV1}) is bounded from above by 
\begin{eqnarray*}
&&\frac{c h f(x) \V(r)}{r^{d+1}} \int_{B_+ \setminus B(0,r/4)_+} \V(\delta(y)) R(\delta(y))\, dy\\
&=& \frac{c h f(x) \V(r)}{r^{d+1}} \int_{r/4}^{r} \rho^{d-1} \V(r-\rho) R(r-\rho) \, d\rho\nonumber\\
&\le&\frac{c h f(x) \V(r)}{r^{2}} \int_{0}^{r}  \V(\rho) R(\rho) \, d\rho\nonumber\\
&\le&\frac{c h f(x) }{r}\nonumber
\end{eqnarray*}
Here in the last step we used the estimate    $\int_{0}^{r}  \V(\rho) R(\rho) \, d\rho\le C \frac{r }{\V(r)}$ from \cite[Proposition 3.5]{BGR2013_1} .
 This gives that $\text{III}$ is bounded from above by $c h f(x)/r$.

Finally we obtain $\text{I} + \text{II} + \text{III}  \le  c h f(x)/r$. Using again the Harnack inequality we get $f(x) \le cf(0)$.

\end{proof}

\begin{lemma}\label{derivative_1} Let $|x|<r<|y|$. Then 
\begin{equation*}
\left|\frac{\partial}{\partial x_1} p_{t}(x-y)\right|\le 6 \pi
\left( r p^{(d+2)}_{t}(r-|x|)+ |y|p^{(d+2)}_{t}(|y|/2) \right) .
\end{equation*}
\end{lemma}
\begin{proof} Since $r-|x|\le  |x-y|\le 3r$ for $|y|\le2r$, and  $|y|/2\le  |x-y|\le 2|y|$ for $|y|>2r $,  by Theorem \ref{dertransition} and radial monotonicity of $p_t^{(d+2)}$, 
we obtain
$$
\left| \frac{\partial}{\partial x_1} p_{t}(x-y) \right| = 
2 \pi |x_1 - y_1| p_t^{(d+2)}(|x-y|) \le 6 \pi r p_t^{(d+2)}(r-|x|)+ 4 \pi |y| p^{(d+2)}_{t}(|y|/2).
$$

 \end{proof}
 
 We define  $$r_{B}(t,x,y)= E^x(p_{t - \tau_D}(X(\tau_D),y), \, t > \tau_D), \quad x,y \in D, \, t > 0.$$
 Recall that $ p_D(t,x,y) = p_t(x-y) - r_{B}(t,x,y)$.
  
\begin{lemma}
\label{rB}
 For any $r \in (0,1]$, $B = B(0,r)$, $t > 0$, $x,y \in B$ we have
$$
\left|\frac{\partial}{\partial x_1} r_{B}(t,x,y)\right| \le f_t(\delta(x),y),
$$
where $f_t:(0,r] \times B \to (0,\infty)$ is a Borel function and $\delta(x)=\delta_B(x)$. For each fixed $t > 0$, $y \in B$ we have $f_t(a,y) \nearrow$ when $a \searrow$ and for each $a \in (0,r]$, $y \in B$ we have $\int_0^{\infty} f_t(a,y) \, dt < \infty$. For each fixed $t > 0$, $ a \in (0,r]$ we have $\int_B f_t(a,y) \, dy < \infty$. For each fixed $a \in (0,r]$ we have $\int_0^{\infty} \int_B f_t(a,y) \, dy \, dt < \infty$.
\end{lemma}
\begin{proof}
We have
\begin{equation}
\label{derivativerB}
\left|\frac{\partial}{\partial x_1} r_{B}(t,x,y)\right|
= \left|\frac{\partial}{\partial x_1} E^y\left[p_{t - \tau_B}(x - X(\tau_B)), t > \tau_B\right]\right|.
\end{equation}
Applying Lemma \ref{derivative_1} we obtain
\begin{eqnarray*}\left|\frac{\partial}{\partial x_1} \left[p_{t - \tau_B}(x - X(\tau_B))\right]\right|&\le& 
 6\pi r p^{(d+2)}_{t - \tau_B}(r-|x|)+  6\pi|X(\tau_B)|p^{(d+2)}_{t- \tau_B}(|X(\tau_B)|/2)).
 \end{eqnarray*}
 Moreover by Corollary \ref{derestimates},
 \begin{eqnarray*}\left|\frac{\partial}{\partial x_1} \left[p_{t - \tau_B}(x - X(\tau_B))\right]\right|&\le& \frac c{(r-|x|)^{d+1}}.\end{eqnarray*}
It follows that we can change the order of $\frac{\partial}{\partial x_1}$ and $E^y$ in (\ref{derivativerB}). We have also shown that  
$$
\left|\frac{\partial}{\partial x_1} r_{B}(t,x,y)\right| \le f_t(\delta(x),y),
$$
where
\begin{eqnarray*}
f_t(a,y) &=& 6\pi E^y\left[ r p_{t - \tau_B}^{(d+2)}(a), t > \tau_B\right] 
 + 6 \pi  E^y\left[ |X(\tau_B)| p_{t - \tau_B}^{(d+2)}(|X(\tau_B)|/2), t > \tau_B\right].
\end{eqnarray*}
Of course, $f_t:(0,r] \times B \to (0,\infty)$, $f_t$ is a Borel function, for each fixed $t > 0$, $y \in B$ we have $f_t(a,y) \nearrow$ when $a \searrow$. We also have
\begin{eqnarray*}
\int_0^{\infty} f_t(a,y) \, dt &=& 8\pi r E^y\left[ \int_{\tau_B}^{\infty} p_{t - \tau_B}^{(d+2)}(a) \, dt \right] \\
&& + 4 \pi  E^y\left[  |X(\tau_B)| \int_{\tau_B}^{\infty} p_{t - \tau_B}^{(d+2)}(|X(\tau_B)|/2) \, dt \right]\\
&=& 6\pi r U^{(d+2)}(a) +
6 \pi  E^y\left[ |X(\tau_B)| U^{(d+2)}(|X(\tau_B)|/2)\right]\\
&\le& 8\pi r U^{(d+2)}(a)+ 8\pi \sup_{\rho>r} \rho U^{(d+2)}(\rho).
\end{eqnarray*}

By \cite[Theorem 16]{G2013} and then by (\ref{psi_star}) this is bounded from above by 
$$
\frac{c r L^2(a)}{a^{d+2}} + \sup_{\rho>r}
 \frac{c L^2(\rho)}{\rho ^{d+1} } < \frac{c r L^2(a)}{a^{d+2}} + \sup_{\rho>r}
 \frac{c \frac{\rho^2}{r^2} L^2(r)}{\rho ^{d+1} } \le \frac{c r L^2(a)}{a^{d+2}} + 
 \frac{c  L^2(r)}{r ^{d+1} },
$$
where $c = c(d)$.  

 It follows that for each fixed $a \in (0,r]$ we have $\int_0^{\infty} \int_B f_t(a,y) \, dy \, dt < \infty$.
\end{proof}

By saying that $\frac{\partial f}{\partial x_i} (x)$ exists we understand that $\lim_{h \to 0} \frac{f(x + h e_i) - f(x)}{h}$ exists and is finite. 

\begin{lemma}
\label{Greenderivative}
 For any $r \in (0,1]$, $B = B(0,r)$, $x,y \in B$, $x \ne y$ there exists
$$
\frac{\partial}{\partial x_1} G_B(x,y)
$$
and
$$
\frac{\partial}{\partial x_1} G_B(x,y) = \int_0^{\infty} \frac{\partial}{\partial x_1} p_B(t,x,y) \, dt.
$$
\end{lemma}
\begin{proof}
Fix $t > 0$, $x,y \in B$, $x \ne y$ and put $s = \frac{|x-y|}{2} \wedge \frac{\delta(x)}{2}$. We will estimate $\frac{\partial}{\partial z_1} p_B(t,z,y)$ for $z \in B(x,s)$. For $z \in B(x,s)$ we have
$$
\frac{\partial}{\partial z_1} p_B(t,z,y) =
\frac{\partial}{\partial z_1} p_{t}(z-y) - \frac{\partial}{\partial z_1} r_B(t,z,y).
$$
We have
$$
\left| \frac{\partial}{\partial z_1} p_{t}(z-y) \right| = 
2 \pi |z_1 - y_1| p_t^{(d+2)}(|z-y|) \le 4 \pi r p_t^{(d+2)}(s)
$$
and
$$
\left|\frac{\partial}{\partial z_1} r_B(t,z,y) \right| \le
f_t(\delta(z),y) \le f_t(s,y),
$$
where $f_t$ is a function defined in Lemma \ref{rB}. We have $\int_0^{\infty}( p_t^{(d+2)}(s) + f_t(s,y)) \, dt < \infty$. This justifies the chage of the derivative and integral in $\frac{\partial}{\partial x_1} \int_0^{\infty} p_B(t,x,y) \, dt$ and implies the assertion of the lemma.
\end{proof}

\begin{proposition}
\label{derGreen}
 For any $r \in (0,1]$, $B = B(0,r)$, $x,y \in B$, $x \ne y$ we  have
$$
\frac{\partial}{\partial x_1} G_B(x,y) \le c \frac{G_B(x,y)}{|x-y| \wedge \delta(x)},
$$
where $\delta(y) = \delta_B(y)$.
\end{proposition}
\begin{proof}
Fix $x,y \in B$, $x \ne y$. Let $s = |x - y| \wedge \delta(x)$. The function $z \to G_B(z,y)$ is harmonic (with respect to the process $X$) for $z \in B(x,s/2)$. By continuity of $z \to G_B(z,y)$, $z \ne y$ we obtain that there exists $\varepsilon > 0$ such that for $|x - z| < \varepsilon$ we have $G_B(z,y) \le 2 G_B(x,y)$. Let $h \in (0,\varepsilon \wedge (s/4))$. By Proposition \ref{Lipschitz} we have
\begin{equation}
\label{quotient}
\left|\frac{G_B(x + h e_1,y) - G_B(x,y)}{h}\right| \le c \frac{G_B\left(x + \frac{1}{2} h e_1,y\right)}{s}.
\end{equation}
Since $h \le \varepsilon$ we obtain $G_B\left(x + \frac{1}{2} h e_1,y\right) \le 2 G_B(x,y)$. Using this, (\ref{quotient}) and Lemma \ref{Greenderivative} we obtain the assertion of the lemma.
\end{proof}

\begin{lemma}
\label{dermeanexit}
 For any $R \in (0,1]$, $B = B(0,R)$, $x\in B$ the partial derivative 
$
\frac{\partial}{\partial x_1} E^x(\tau_B)
$
exists. Moreover it equals to $0$ for $x=0$. 
\end{lemma}
\begin{proof}  For the   first part of the proof consider the case when $p_t^{(d+2)}(0)= \|p_t^{(d+2)}\|_{\infty}<\infty$, where $\|\cdot\|_{\infty}$ denotes the supremum norm.
Fix $z \in B$ and put $4r = \delta_B(z)$. We will show that for any $x \in B(z,r)$, $\frac{\partial}{\partial x_1} E^x(\tau_B)$ exists. 

For any $t > 0$, $x \in B(z,r)$ put $g_t(x) = \int_{B(z,2r)} p_t(x-y) \, dy$. For any $t > 0$, $x \in B(z,r)$, $y \in B(z,2r)$ we have
$$
\left| \frac{\partial}{\partial x_1} p_t(x-y)\right| = 
2 \pi \left|(x_1 - y_1) p_t^{(d+2)}(x-y)\right| \le 6\pi r \|p_t^{(d+2)}\|_{\infty} < \infty.
$$
It follows that for any $t > 0$, $x \in B(z,r)$ we have
$$
\frac{\partial g_t}{\partial x_1}(x) = 
\int_{B(z,2r)} \frac{\partial p_t}{\partial x_1}(x-y) \, dy
= -2 \pi \int_{B(z,2r)} (x_1 - y_1) p_t^{(d+2)}(x-y) \, dy.
$$
In particular, $\frac{\partial g_t}{\partial x_1}(x)$ exists for any $x \in B(z,2r)$. 

For any $t > 0$, $x \in B(z,r)$ we also have
$$
\frac{\partial g_t}{\partial x_1}(x) 
= -2 \pi \int_{B(x,r)} (x_1 - y_1) p_t^{(d+2)}(x-y) \, dy
-2 \pi \int_{B(z,2r) \setminus B(x,r)} (x_1 - y_1) p_t^{(d+2)}(x-y) \, dy.
$$
By radial symmetry of $p_t^{(d+2)}$ the first integral vanishes, hence finally 
\begin{equation}\label{derivative_g}
\frac{\partial g_t}{\partial x_1}(x) = 
-2 \pi \int_{B(z,2r) \setminus B(x,r)} (x_1 - y_1) p_t^{(d+2)}(x-y) \, dy.
\end{equation}

To remove the assumption that $p_t^{(d+2)}(0)= \|p_t^{(d+2)}\|_{\infty}<\infty$ we consider   the process with the symbol $\psi(\xi)+\epsilon| \xi|, \, \epsilon>0$, that is we add to $X$ an independent Cauchy process multiplied by $\epsilon>0$. This new process satisfies all the assumptions needed in Theorem \ref{dertransition}  to construct its $d+2$-dimensional corresponding variant. Moreover this $d+2$ - dimensional process  has uniformly bounded transition densities for each $t$. Hence we can repeat all the above steps and then pass with $\epsilon\to 0$ to arrive at (\ref{derivative_g}) in this case. The passage is easily justyfied by observing that the integrand in (\ref{derivative_g}) is bounded by $3r p_t^{(d+2)}(r)$. We leave the details to the Reader.

Hence for any 
$t > 0$, $x \in B(z,r)$ we have
\begin{eqnarray*}
\left|\frac{\partial g_t}{\partial x_1}(x)\right| &=&
2 \pi \left| \int_{B(z,2r) \setminus B(x,r)} (x_1 - y_1) p_t^{(d+2)}(x-y) \, dy \right| \\
&\le& 6 \pi   \int_{B(x,3r) \setminus B(x,r)} |x-y| p_t^{(d+2)}(x-y) \, dy\\
&\le& \frac{c}{r}  \int_r^{3r} \rho^{d+1} p_t^{(d+2)}(\rho) \, d\rho\\
&=& \frac{c}{r} \int_{B^{*}(0,3r) \setminus B^{*}(0,r)} p_t^{(d+2)}(y) \, dy,
\end{eqnarray*}
where $c = c(d)$ and $B^{*}(0,u) = \{y \in \R^{d+2}:\, |y| < u\}$, $u > 0$.

Now for any $t > 0$, $x \in B(z,r)$ put $h_t(x) = \int_{B(z,2r)} r_B(t,x,y) \, dy$. By Lemma \ref{rB} for any $t > 0$, $x \in B(z,r)$, $y \in B(z,2r)$ we have
$$
\left| \frac{\partial}{\partial x_1} r_B(t,x,y)\right| \le f_t(\delta(x),y) \le f_t(3r,y).
$$
For any $t > 0$ Lemma \ref{rB} also gives $\int_B f_t(3r,y) \, dy < \infty$. Hence for any $t > 0$, $x \in B(z,r)$ we have
$$
\frac{\partial h_t}{\partial x_1}(x) = 
\int_{B(z,2r)} \frac{\partial}{\partial x_1} r_B(t,x,y) \, dy
$$
and
$$
\left| \frac{\partial h_t}{\partial x_1}(x) \right| \le 
\int_{B(z,2r)} f_t(3r,y) \, dy < \infty.
$$
Now for any $x \in B(z,r)$ put
\begin{equation}
\label{Arep}
A(x) = \int_{B(z,2r)} G_B(x,y) \, dy
= \int_0^{\infty} (g_t(x) - h_t(x)) \, dt.
\end{equation}
Now for any $t > 0$, $x \in B(z,r)$ we have
$$
\left| \int_0^{\infty} \left(\frac{\partial g_t}{\partial x_1}(x) - \frac{\partial h_t}{\partial x_1}(x)\right) \, dt \right|
\le \frac{c}{r} \int_{B^{*}(0,3r) \setminus B^{*}(0,r)} p_t^{(d+2)}(y) \, dy
+ \int_{B(z,2r)} f_t(3r,y) \, dy.
$$
Next, 
$$
\frac{c}{r} \int_0^{\infty} \int_{B^{*}(0,3r) \setminus B^{*}(0,r)} p_t^{(d+2)}(y) \, dy \, dt =\frac{c}{r}\int_{B^{*}(0,3r) \setminus B^{*}(0,r)}U^{(d+2)}(y)\, dy <\infty.
$$
By Lemma \ref{rB} we have
$$
\int_0^{\infty} \int_{B(z,2r)} f_t(3r,y) \, dy \, dt < \infty.
$$
Using this and (\ref{Arep}) we obtain that $\frac{\partial}{\partial x_1} A(x)$ exists for any $x \in B(z,r)$.

Now for any $x \in B(z,r)$ put 
$$
B(x) = \int_{B \setminus B(z,2r)} G_B(x,y) \, dy.
$$
By Proposition \ref{derGreen} for any $x \in B(z,r)$, $y \in B \setminus B(z,2r)$  we get
$$
\left| \frac{\partial}{\partial x_1} G_B(x,y)\right| \le c \frac{G_B(x,y)}{r} \le c \frac{G_B(0,y)}{r},
$$
where the last step follows by applying the Harnack inequality to $G_B(\cdot,y)$.    Since $G_B(0,y)$ is integrable we obtain that 
 $\frac{\partial}{\partial x_1} B(x)$ exists for any $x \in B(z,r)$.  Finally the function $x\to E^x(\tau_B)$ is symmetric, so its partial derivative exists  at  $x=0$ and must be equal to $0$.  The proof is completed. 
\end{proof}

\begin{proposition}
\label{existence}
  Let $f:\R^d \to [0,\infty)$ be harmonic in an open nonempty set $D \subset \R^d$. Then $\frac{\partial f}{\partial x_i}(x)$ exists for any $i = 1,\ldots,d$ and $x \in D$.
\end{proposition}
\begin{proof}
The proof resembles to some extent the proof of Lemma 4.3 in \cite{K2013}. We may assume that $i = 1$. 
Fix $z \in D$ and choose $r \in (0,1]$ such that $B(z,3r) \subset D$. Put $B = B(z,r)$. Then we have 
$$
f(x) = \int_B G_B(x,y) \int_{B^c} f(w) \nu(y-w) \, dw \, dy, \quad \quad x \in B.
$$
Put 
$$
g(y) = \int_{B^c} f(w) \nu(y - w) \, dw, \quad \quad y \in B
$$
and $u(y) = g(y) - g(z)$, $y \in B$. We have $f(x) = G_Bg(x)$. By Lemma \ref{Lipschitzg} we obtain
\begin{equation}
\label{uest}
|u(y)| \le \frac cr |y - z| g(z), \quad \quad y \in B(z,r/2).
\end{equation}
 Let $h \in (-r/8,r/8)$. We have
\begin{eqnarray*}
G_Bg(z+h e_1) - G_Bg(z) &=& (G_B 1_B(z + h e_1) - G_B 1_B(z)) g(z)\\
&& + G_B u(z+ h e_1) - G_B u(z).
\end{eqnarray*}
By  Lemma \ref{dermeanexit} we get
$$
\lim_{h \to 0} \frac{1}{h} (G_B 1_B(z + h e_1) - G_B 1_B(z)) g(z) = 
g(z) \frac{\partial}{\partial z_1} G_B 1_B(z) = 0.
$$
We also have 
\begin{eqnarray*}
&& \frac{1}{h} (G_B u(z+ h e_1) - G_B u(z))
= \int_B \frac{1}{h} (G_B (z+ h e_1,y) - G_B (z,y))u(y) \, dy\\
&& = \int_{B(z,2|h|)} + \int_{B \setminus B(z,2|h|)}  = \text{I} + \text{II}.
\end{eqnarray*}
By (\ref{uest}) and next by the bounded convergence theorem
\begin{eqnarray*}
|\text{I}| &\le& c g(z) \int_{B(z,2|h|)} \left(G_B(z+ h e_1,y) + G_B(z, y)\right) \, dy.\\
&\le& c g(z) \int_{B(0,3|h|)} G_{B(0,2r)}(0,y) \, dy\to 0,\quad h \to 0.
\end{eqnarray*}

 Applying Proposition \ref{derGreen}   we have for any $y \in B \setminus B(z,2|h|)$,
\begin{eqnarray*} \left|\frac{1}{h} (G_B u(z+ h e_1) - G_B u(z))\right|&=& \left|\frac{\partial G_B}{\partial z_1}(z + h \theta e_1,y)\right|\\&\le& c \frac{G_B(z + h \theta e_1,y)}{|z-y| \wedge r}\\
&\le& c \frac{G_B(z ,y)}{|z-y| \wedge r},\end{eqnarray*}
where $0\le \theta\le 1$ and the last inequality follows from the Harnack principle. Next we show that  $ \frac{G_B(z ,y)}{|z-y| \wedge r}|u(y)|$ is intgrable over $B$. By (\ref{uest}), for $y \in B(z,r/2)$, we have 
$$ \frac{G_B(z ,y)}{|z-y| \wedge r}|u(y)|\le  c\frac{G_B(z ,y)}{r|z-y| }g(z) |z-y|= c \frac{G_B(z ,y)}{r}g(z),$$
while for $y \in B\setminus  B(z,r/2)$ we obtain 
$$ \frac{G_B(z ,y)}{|z-y| \wedge r}|u(y)|\le  \frac{G_B(z ,y)}{|z-y| \wedge r}(g(z)+g(y))\le  2\frac{G_B(z ,y)}{r}(g(z)+g(y)).$$

Of course, $y \to  G_B(z,y)g(y) =  G_B(z,y) \int_{B^c} f(w) \nu(y-w) \, dw$ is an integrable function on $B$.
This implies 
$$
\lim_{h \to 0} \text{II} = \int_{B } \frac{\partial}{\partial z_1}G_B(z,y) u(y) \, dy
$$
and finishes the proof of the proposition.
\end{proof}
\section{Examples}

The processes in the first 3 examples are subordinate Brownian motions in $\R^d$, i.e. $X_t = B_{S_t}$ where $B$ is the Brownian motion in $\R^d$ (with a generator $\Delta$) and $S$ is an independent subordinator with the Laplace exponent $\phi$.
 
\begin{example} We assume that the Levy measure of the subordinator $S$  is infinite, $\phi$ is a complete Bernstein function and it satisfies
$$
c_1 \lambda^{\alpha/2} \ell(\lambda) \le \phi(\lambda) \le c_2 \lambda^{\alpha/2} \ell(\lambda), \quad \quad \lambda \ge 1,
$$
where $0 < \alpha < 2$, $\ell$ varies slowly at infinity, i.e. $\forall x > 0$ $\lim_{\lambda \to \infty} \frac{\ell(\lambda x)}{\ell(\lambda)} = 1$.
The process $X$ satisfies assumptions (A).

In particular, one of the processes satisfying the above conditions is {\it{the relativistic process}} in $\R^d$ with the Laplace exponent $\phi(\lambda) = \sqrt{\lambda + m^2} - m$, $m > 0$ and a generator $m - \sqrt{m^2 - \Delta}$, (see \cite{C1989}, \cite{R2002}, \cite{KS2006}). The generator of this process is called the relativistic Hamiltonian and it is used in some models of mathematical physics (see e.g. \cite{LS2010}).
\end{example}
\begin{proof}
It is clear that assumptions (H4), (H10) are satisfied. The fact that (H7) holds it is stated in Example 4 in \cite{KM2012}. Hence assumptions (A3) are satisfied.
\end{proof}

\begin{example}
Let $\phi(\lambda) = \log(1 + \lambda^{\beta/2})$, $\beta \in (0,2]$. The process $X$ is called {\it{the geometric stable processes}} and it satisfies assumptions (A).
\end{example}
\begin{proof}
One can directly check that (H10) is satisfied (see also Example 1 in \cite{KM2012}), the fact that (H7) holds is well known, (H4) is obvious. Hence assumptions (A3) are satisfied.
\end{proof}

\begin{example}
Let $\phi(\lambda) = \frac{\lambda}{\log(1+\lambda)} - 1$. The process $X$ is sometimes called {\it{the conjugate to the variance gamma process}} and it satisfies assumptions (A).
\end{example}
\begin{proof} (H0) is clear, (H2) follows from \cite{M2013} (for $d \ge 3$) and \cite[Example 3]{G2013} ($d \ge 1$). (H1) is implied by two conditions which hold for the density of the L\'evy measure of the subordinator $\nu_S$ (see the proof of Proposition 3.5 \cite{KimSongVondracek2012}),
\begin{itemize}
\item[(a)] For any $K>0$ there is $c=c(K)$ such that 
 $$\nu_S(r)\le c \nu_S(2r), \ 0<r<K.$$
\item[(b)] There exists $C$ such that 
$$\nu_S(r)\le C \nu_S(r+1), \ r\ge 1.$$
\end{itemize}
From the estimates of $\nu_S(r)$ obtained in \cite{M2013} we infer that (a) holds, while (b) is implied by the fact that $\phi(\lambda)$ is a complete Bernstein function \cite[Lemma 2.1 ]{KimSongVondracek2012}.

\end{proof}

The process in the next example is not a subordinate Brownian motion.
\begin{example} Let $\{X_t\}$ be the pure-jump isotropic L{\'e}vy process in $\R^d$ with the L{\'e}vy measure $\nu(dx) = \nu(|x|) \, dx$ given by the formula
\[ \nu(r) = \left\{              
\begin{array}{ll}  \mathcal{A}_{d,\alpha} r^{-d-\alpha}& \text{for} \quad r \in (0,1]\\
c_1 e^{-c_2 r} & \text{for} \quad r \in (1,\infty)          \end{array}       
\right. \]
where $\mathcal{A}_{d,\alpha} r^{-d-\alpha}$ is the L{\'e}vy density  for the symmetric $\alpha$-stable process in $\R^d$ with the characteristic exponent $\psi(x) = |x|^{\alpha}$, $\alpha \in (0,2)$ and $c_1 = \mathcal{A}_{d,\alpha} e^{d+\alpha} > 0$, $c_2= d + \alpha >0$ are chosen so that $\nu(r) \in C^1(0,\infty)$. $X$ satisfies assumptions (A).
\end{example}
\begin{proof}
(H0) is obvious and (H1) is easy to check. (WLSC) holds for $\psi(\xi) = \int_{\R^d} (1 - \cos\langle \xi,x \rangle) \, \nu(dx)$ because the characteristic exponent $\psi$ for $X$ behaves for large $\xi$ like the characteristic exponent for the symmetric $\alpha$-stable process. Hence (H3) holds, so assumptions (A1) are satisfied.
\end{proof}

Now we show an example of a harmonic function for some L{\'e}vy process for which the gradient does not exist at some point. The process is a pure-jump, isotropic unimodal L{\'evy} process, which L{\'e}vy measure does not satisfy the assumption that $-\nu'(r)/r$ is nonincreasing (cf. (H1)).
\begin{example}
Let $X$ be a pure-jump, L{\'evy} process in $\R$ which L{\'e}vy measure $\nu(dx) = \nu(x) \, dx$ has the density given by the formula
\[ \nu(x) = \left\{              
\begin{array}{ll}  \calA_{\alpha} |x|^{-1-\alpha} & \text{for} \quad |x| \in (0,1], \\
\calA_{\alpha} (1 - (|x| - 1)^{\gamma}) & \text{for} \quad |x| \in (1,2], \\
0 & \text{for} \quad |x| \in (2,\infty),          \end{array}       
\right. \] 
where $\alpha \in (0,1/2)$, $\gamma \in (1/2,1)$,  $\alpha + \gamma <1$, $\calA_{\alpha} |x|^{-1-\alpha}$ is the density of the L{\'e}vy measure for  $\alpha$-stable process in $\R$ with the characteristic exponent $\psi(x) = |x|^{\alpha}$.

Note that $\nu(x)$ satisfies $\nu(-x) = \nu(x)$, it is continuous and nonincreasing on $(0,\infty)$. It follows that the process $X$ is isotropic unimodal. 

Let $B = (-1/2,1/2)$ and let us define the function $f$ by
\[ f(z) = \left\{              
\begin{array}{ll} 
(z - 1)^{-\beta}  & \text{for} \quad z \in (1,2), \\
0 & \text{for} \quad z \notin (B \cup (1,2)), \\
E^z(f(X(\tau_B))) & \text{for} \quad z \in B,          \end{array}       
\right. \]
where $\beta \in (0,1)$, $\alpha - \beta + \gamma < 0$.

Then $f$ is harmonic on $B$ with respect to the process $X$ but $f'(0)$ does not exist.
\end{example}
\begin{proof}
Note that for any $y \in B_+$ we have $\tilde{P}^y(\tilde{X}(\tau_{B_+}) \in [1/2,5/2]) = 1$.
By the arguments used in the proof of Proposition \ref{Lipschitz} we get
\begin{equation}
\label{fdifference}
f(x) - f(-x) = \int_{B_+} \tilde{G}_{B_+}(x,y) (g(y) - g(-y)) \, dy,
\end{equation}
where $B_+ = (0,1/2)$, $x \in B_+$, $g(y) = \int_{B^c} \nu(y-z) f(z) \, dz$, $y \in \R$.
For $y \in B_+$ we have
\begin{eqnarray*}
g(y) &=&
\calA_{\alpha} \int_1^{1+y} (z-y)^{-1-\alpha} (z-1)^{-\beta} \, dz
+ \calA_{\alpha} \int_{1+y}^2 (1 - (z - y - 1)^{\gamma}) (z-1)^{-\beta} \, dz \\
&\ge& \calA_{\alpha} \int_1^{2} (z-1)^{-\beta} \, dz
- \calA_{\alpha} \int_{1+y}^2 (z - y - 1)^{\gamma} (z-1)^{-\beta} \, dz
\end{eqnarray*}
and
\begin{equation*}
g(-y) = \calA_{\alpha} \int_1^{2} (z-1)^{-\beta} \, dz
- \calA_{\alpha} \int_{1}^2 (z + y - 1)^{\gamma} (z-1)^{-\beta} \, dz.
\end{equation*}
Hence for $y \in B_+$ we have
\begin{equation}
\label{gdifference}
g(y) - g(-y) \ge \calA_{\alpha} \int_{1}^{1+y} (z + y - 1)^{\gamma} (z-1)^{-\beta} \, dz \ge c y^{1 - \beta + \gamma},
\end{equation}
where $c = c(\alpha,\beta,\gamma)$.

Now we need to use the inequality, which we justify later: 
\begin{equation}
\label{Greentildeineq}
\tilde{G}_{B_+}(x,y) \ge c_1 x y^{\alpha - 2} - c_2 x, \quad \quad x \in (0,1/16), \, \, y \in (2x,1/4),
\end{equation}
where $c_1 = c_1(\alpha, \gamma)$, $c_2 = c_2(\alpha, \gamma)$.
Using (\ref{fdifference}), (\ref{gdifference}) and (\ref{Greentildeineq}) we get for $x \in (0,1/16)$
$$
f(x) - f(-x) \ge c_1 x \int_{2x}^{1/4} y^{\alpha - \beta + \gamma - 1} \, dy - c_2 x
\ge c_1 x^{1 + \alpha - \beta + \gamma} - c_2 x,
$$
where $c_1 = c_1(\alpha,\beta,\gamma)$, $c_2 = c_2(\alpha,\beta,\gamma)$. By our assumptions on $\alpha$, $\beta$, $\gamma$ we obtain $1 + \alpha - \beta + \gamma \in (0,1)$. It follows that $f'(0)$ does not exist.

What remains is to prove (\ref{Greentildeineq}). By the definition of $\tilde{G}_{B_+}(x,y)$ we have
\begin{eqnarray}
\label{tildeGreen011}
\tilde{G}_{B_+}(x,y)
&\ge& \int_0^1 \tilde{p}_{B_+}(t,x,y) \, dt \\
\label{tildeGreen012}
&=& \int_0^1 \tilde{p}(t,x,y) - \tilde{E}^y\left( \tilde{p}(t-\tau_{B_+},x,\tilde{X}(\tau_{B_+})), \, \tau_{B_+} < t \right) \, dt. 
\end{eqnarray}
Let $\psi$ be the characteristic exponent of $X$. By the formula for $\nu$ we obtain that $\psi$ satisfies WLSC, so we can use Proposition \ref{dertransitionold}. By this proposition there exists a L{\'e}vy process in $\R^3$ with the characteristic exponent $\psi^{(3)}(\xi) = \psi(|\xi|)$, $\xi \in \R^3$ and the continuous transition density $p_t^{(3)}(x) = p_t^{(3)}(|x|)$, $x \in \R^3$, $t > 0$ satisfying $p_t^{(3)}(r) = (-1/(2\pi r)) p_t'(r)$, $r > 0$. It follows that
\begin{equation}
\label{tildeptxy}
\tilde{p}(t,x,y) = -2 x p'_t(y + \xi) = 4 \pi x (y + \xi)p_t^{(3)}(y + \xi),
\end{equation}
\begin{equation}
\label{tildeptxytau}
\tilde{p}(t-\tau_{B_+},x,\tilde{X}(\tau_{B_+})) = -2 x p'_{t - \tau_{B_+}}(\tilde{X}(\tau_{B_+}) + \xi) = 4 \pi x (\tilde{X}(\tau_{B_+}) + \xi) p_{t-\tau_{B_+}}^{(3)}(\tilde{X}(\tau_{B_+}) + \xi),
\end{equation}
where $\xi \in(-x,x)$.

Let $D \subset \R$ be a bounded, open, nonempty, symmetric ($D = -D$) set. For any $s > 0$, $y \in D_+$, $z \in \left((\overline{D})^c\right)_+$ put 
\begin{equation}
\label{hformula1}
h_{D_+}(y,s,z) = \int_{D_+} \tilde{p}_{D_+}(s,y,w) \tilde{\nu}(w-z) \, dw.
\end{equation}
By (\ref{tildeIW}) and standard arguments (see e.g. Proposition 2.5 in \cite{KS2006}) $h_{D_+}(y,s,z)$ provides the distribution of $(\tau_{D_+}, \tilde{X}(\tau_{D_+}))$ if the process $\tilde{X}$ starts from $y \in D_+$.

Note that 
\begin{equation*}
h_{D_+}(y,s,z) \le \sup_{w\in D_+}{\nu}(w-z)  \int_{D_+} \tilde{p}_{D_+}(s,y,w)  \, dy\le \sup_{w\in D_+}{\nu}(w-z),
\end{equation*}

\begin{eqnarray}
\label{new1}
&&\tilde{E}^y\left( \tilde{p}(t-\tau_{B_+},x,\tilde{X}(\tau_{B_+})), \, \tau_{B_+} < t, \, \tilde{X}(\tau_{B_+})>3/4 \right)\\ 
\label{new2}
&=& \int_0^t\int_{3/4}^{5/2} h_{D_+}(y,s,z)\tilde{p}(t-s,x,z) dzds \\
\label{new3}
& \le& c \int_0^t\int_{3/4}^{5/2}\tilde{p}(s,x,z) dzds =  c \int_0^t\int_{3/4}^{5/2}\int_{-x}^{x}\frac d{dw}{p}(s,w-z)dw dzds\\
\label{new4}
&=& c \int_0^t\int_{-x}^{x}\int_{3/4}^{5/2} (z-w){p}^{(3)}(s,w-z)dzdwds \le ctx,
\end{eqnarray}
where $c = c(\alpha,\gamma)$.

Let $Y$ be the symmetric $\alpha$-stable process in $\R^3$, with the L{\'e}vy measure $\nu^{(3)}_Y(dx) = \nu^{(3)}_Y(|x|)\, dx$ and the transition density $q_t^{(3)}(x)$. Now we need to use the inequality, which we justify later: 
\begin{equation}
\label{stable}
|p_t^{(3)}(y) - e^{-mt} q_t^{(3)}(y)| \le c, \quad \quad t \in (0,1), \, y \in B(0,13/16),  
\end{equation}
where $c = c(\alpha,\gamma)$, $m = m(\alpha,\gamma) \in (-\infty,\infty)$. It is well known that 
\begin{equation}
\label{stabletransition}
c_1 \min(t^{-3/\alpha},t |y|^{-3-\alpha}) \le q_t^{(3)}(y) \le c_2 \min(t^{-3/\alpha},t |y|^{-3-\alpha}), \quad t > 0, \, y \in \R^3,
\end{equation}
where $c_1 = c_1(\alpha)$, $c_2 = c_2(\alpha)$. 

By (\ref{tildeptxy}) and (\ref{stable}) we get for $x \in (0,1/16)$, $y \in  (2x,1/4)$, $t \in (0,1]$
\begin{equation}
\label{tildeptxy1}
\tilde{p}(t,x,y) \ge c_1 x y q_t^{(3)}(y) - c_2 x,
\end{equation}
where $c_1 = c_1(\alpha,\gamma)$, $c_2 = c_2(\alpha,\gamma)$.

By  (\ref{tildeptxytau}), (\ref{stable}) and (\ref{stabletransition}) we get for $x \in (0,1/16)$, $y \in  (2x,1/4)$, $t \in (0,1]$
\begin{equation}
\label{tildeptxytau1}
\tilde{E}^y\left( \tilde{p}(t-\tau_{B_+},x,\tilde{X}(\tau_{B_+})), \, \tau_{B_+} < t, \, \tilde{X}(\tau_{B_+})\le 3/4 \right) \le c x,
\end{equation}
where $c = c(\alpha,\gamma)$.

By (\ref{tildeGreen011} - \ref{tildeGreen012}), (\ref{tildeptxy1}), (\ref{new1} - \ref{new4}) and (\ref{tildeptxytau1}) we get for $x \in (0,1/16)$, $y \in  (2x,1/4)$
$$
\tilde{G}_{B_+}(x,y)
\ge c_1 x y \int_0^1 q_t^{(3)}(y) \, dt - c_2 x \ge c_1 x y^{\alpha -2} - c_2 x,
$$
where $c_1 = c_1(\alpha,\gamma)$, $c_2 = c_2(\alpha,\gamma)$. This gives (\ref{Greentildeineq}).

What remains is to show (\ref{stable}). Denote by $\nu^{(3)}(dx) = \nu^{(3)}(|x|)$ the L{\'e}vy meausure of $X^{(3)}$. By (\ref{derivativeptr}) for $r \in (0,\infty) \setminus \{1,2\}$ we get 
\begin{equation}
\label{nu3}
\nu^{(3)}(r) =
1_{(0,1)}(r) \nu_Y^{(3)}(y) - 1_{(1,2)}(r) \frac{\nu'(r)}{2 \pi r} = \nu_Y^{(3)}(y) + \mu(r),
\end{equation}
where $\supp(\mu) = [1,\infty)$, $\mu(r)$ has the singularity at $r = 1$ of the type $(r-1)^{\gamma -1}$ and $\mu(r)$ changes the sign on $(1,\infty)$. Put $\mu(x) = \mu(|x|)$, $x \in \R^3$, $m = \int_{\R^3} \mu(x) \, dx$, $M = \int_{\R^3} |\mu(x)| \, dx$. We have $M < \infty$, $m \in (-\infty,\infty)$. From (\ref{nu3}) it follows that for $t >0$
$$
p_t^{(3)} = q_t^{(3)}*\left(e^{-tm} \sum_{n = 0}^{\infty} \frac{t^n \mu^{(*n)}}{n!}\right)
$$
so
\begin{equation}
\label{p3q3}
p_t^{(3)}(x) = e^{-tm} \left(q_t^{(3)}(x) + q_t^{(3)} *  \sum_{n = 1}^{\infty} \frac{t^n \mu^{(*n)}}{n!}(x)\right), \quad x \in \R^3.
\end{equation}
By (\ref{stabletransition}) and the fact that $\supp(\mu) \subset B^c(0,1)$ we get $\sup_{x \in B(0,13/16)} |q_t^{(3)}*\mu(x)| \le M_1 < \infty$. Since $\gamma \in (1/2,1)$ we get that the function $\mu \in L^2(\R^3)$ so $||\mu^{(*2)}||_{\infty} \le M_2 < \infty$. It follows that $||\mu^{(*n)}||_{\infty} \le M_2 M^{n-2}$, $n \ge 2$. Hence for $t \in (0,1]$ and $x \in B(0,1/2)$ we have $\left| q_t^{(3)} *  \sum_{n = 1}^{\infty} \frac{t^n \mu^{(*n)}}{n!}(x) \right| \le M_1 + M_2 + M_2 \sum_{n = 2}^{\infty} \frac{t^n M^{n-2}}{n!} \le M_1 + M_2 + M_2 e^{M}$. This and (\ref{p3q3}) gives (\ref{stable}).
\end{proof}

{\bf{ Acknowledgements.}} We would like to thank  Z.-Q.
Chen and T. Grzywny for discussions on the problem
treated in this paper. T. Kulczycki is grateful for the hospitality of the Institute of Mathematics, Polish Academy of Sciences, the branch in Wroc{\l}aw, where part of this paper was written. 

Z.-Q. Chen informed us that he and his student T. Yang obtained independently gradient estimates of harmonic functions for $(-\Delta)^{\alpha/2} + a^{\beta} (-\Delta)^{\beta/2}$, $0 < \beta < \alpha < 2$, $a \in (0,1]$. Their estimates are uniform with respect to $a \in (0,1]$.


\begin{thebibliography}{99}
\bibliographystyle{plain}

\bibitem{BG1968} R. M. Blumenthal, R. K. Getoor, \emph{Markov Processes and Their Potential Theory}, Pure Appl. Math., Academic Press, New York (1968).
 
\bibitem{BGR2013} K. Bogdan, T. Grzywny, M. Ryznar, \emph{Density and tails of unimodal convolution semigroups},  arXiv:1305.0976 (2013).

\bibitem{BGR2013_1} K. Bogdan, T. Grzywny, M. Ryznar, \emph{Barriers,
exit time and survival probability for unimodal L\'evy processes},  arXiv:1307.0270 (2013).

\bibitem{BJ2007} K. Bogdan, T. Jakubowski \emph{Estimates of Heat Kernel of Fractional Laplacian Perturbed by Gradient Operators}, Comm. Math. Phys.
271, (2007) 179-198.

\bibitem{BKN2002} K. Bogdan, T. Kulczycki, A. Nowak, \emph{Gradient estimates for harmonic and $q$-harmonic functions of symmetric stable processes}, Illinois J. Math. 46 (2002), 541-556.

\bibitem{BKK2013} K. Bogdan, T. Kumagai, M. Kwa{\'s}nicki, \emph{Boundary Harnack inequality for Markov processes with jumps}, to appear in Trans. Amer. Math Soc. (2013). 

\bibitem{C1989} R. Carmona, \emph{Path integrals for relativistic Schr{\"o}dinger operators}, Lect. Notes in Phys. 345 (1989), 65-92. 

\bibitem{CKS2012} Z.-Q. Chen, P. Kim, R. Song, \emph{Dirichlet heat kernel estimates for fractional Laplacian with gradient perturbation}, Annals of Probab. 40 (2012), 2483-2538.

\bibitem{CKS2013} Z.-Q. Chen, P. Kim, R. Song, \emph{Dirichlet Heat Kernel Estimates for Rotationally Symmetric L\'evy processes}, arXiv:1303.6449 (2013).

\bibitem{CKS2010} Z.-Q. Chen, P. Kim, R. Song, \emph{Heat kernel estimates for Dirichlet fractional Laplacian}
J. Euro. Math. Soc. 12 (2010), 1307-1329. 
 
\bibitem{CZ1995} K.  L.  Chung, Z.  Zhao \emph{From Brownian
motion to Schr\"odinger's equation}, Springer-Verlag,
Berlin Heidelberg (1995).

\bibitem{CZ1990} M. Cranston, Z. Zhao, \emph{Some regularity results and eigenfunction estimates for the Schr{\"o}dinger operator}, Diffusion processes and related problems in analysis, I (Evanston, IL, 1989), Progr. Probab., vol. 22, Birkh{\"a}user  Boston, Boston, MA, (1990), pp. 139-146.

\bibitem{GT2013} L. Grafakos, G. Teschl, \emph{On Fourier Transforms of Radial Functions and Distributions}, Journal of Fourier Analysis and Applications 19 (2013), 167-179.

\bibitem{G2013} T. Grzywny, \emph{On Harnack inequality and H{\"o}lder regularity for isotropic unimodal L{\'e}vy processes}, Potential Anal., published online 10 July 2013, doi 10.1007/s11118-013-9360-y.

\bibitem{IW1962}
N. Ikeda, S. Watanabe, \emph{On some relations between the harmonic measure and the Levy measure for a certain class of Markov processes}, J. Math. Kyoto Univ. 2  (1962), 79-95.
  
\bibitem{KSz2013} K. Kaleta, P. Sztonyk, \emph{Estimates of transition densities and their derivatives for jump L{\'e}vy processes}, arXiv:1307.1302.

\bibitem{KM2012} P. Kim, A. Mimica, \emph{Harnack inequalities for subordinate Brownian motions}, Electronic
Journal of Probability 17 (2012), \# 37.

\bibitem{KimSongVondracek2012} P. Kim, R. Song, Z. Vondracek, \emph{Potential theory of subordinate Brownian motions revisited}, Stochastic analysis and applications to finance, 243-290, Interdiscip. Math. Sci., 13, World Sci. Publ., Hackensack, NJ, 2012.

\bibitem{KimSongVondracek2012a} P. Kim, R. Song, Z. Vondracek, \emph{Two-sided Green function estimates for killed subordinate Brownian motions}, Proc. London Math. Soc. 104 (2012), 927-958.

\bibitem{K2013}  T. Kulczycki, \emph{Gradient estimates of q-harmonic functions of fractional Schrödinger operator},  Potential Anal. 39 (2013), 69-98, 

\bibitem{KS2006} T. Kulczycki, B. Siudeja, \emph{Intrinsic ultracontractivity of the Feynman-Kac semigroup for relativistic stable processes}, Trans. Amer. Math. Soc. 358 (2006), 5025-5057.

\bibitem{LS2010} E. H. Lieb, R. Seiringer, \emph{The stability of matter in quantum mechanics}, Cambridge University Press, Cambridge, (2010). 

\bibitem{Millar1975} P. W. Millar, \emph{First passage distributions of processes with independent increments}, Ann. Probability 3 (1975), 215-233. 

\bibitem{M2013} A. Mimica, \emph{Harnack inequality and H{\"o}lder regularity estimates for a L{\'e}vy process with small jumps of high intensity}, J. Theor. Probab. 26 (2013), 329-348.

\bibitem{R2002} M. Ryznar, \emph{Estimates of Green function for relativistic $\alpha$-stable process}, Potential Anal. 17 (2002), 1-23.

\bibitem{Sato1999} K. Sato, \emph{L{\'e}vy processes and infinitely divisible distributions}, volume 68 of Cambridge Studies in Advanced Mathematics, Cambridge University Press, Cambridge (1999).

\bibitem{S1998} R. Schilling, \emph{Growth and H{\"o}lder conditions for the sample paths of Feller processes}, Probab. Theory Relat. Fields 112 (1998), 565-611.

\bibitem{SSzW2012} R. Schilling, P. Sztonyk, J. Wang, \emph{Coupling property and gradient estimates of L{\'e}vy processes via the symbol}, Bernoulli 18 (2012), 1128-1149.

\bibitem{SchillingSongVondracek2010} R. Schilling, R. Song, Z. Vondracek, \emph{Bernstein functions}, Walter de Gruyter , Berlin, second edition, (2012), Theory and applications. 

\bibitem{S2013} L. Silvestre, \emph{On the differentiability of the solution to an equation with drift and fractional diffusion}, Indiana Univ. Math. J.  61 (2012), 557-584. 

\bibitem{Sztonyk2000} P. Sztonyk, \emph{On harmonic measure for L\'evy processes}, Probab. Math. Statist. 20 (2000), 383-390.  

\bibitem{Sz2010} P. Sztonyk,  \emph{Regularity of harmonic functions for anisotropic fractional Laplacians}, Math. Nachr. 283 (2010), 289-311. 

\bibitem{W1983} T. Watanabe, \emph{The isoperimetric inequality for isotropic unimodal L{\'e}vy processes}, Z. Wahrhrsch. Verw. Gebiete 63 (1983), 487-499.

\end{thebibliography}
\end{document}